\documentclass[a4paper,reqno]{amsart}

\usepackage{amssymb}
\usepackage{fullpage}
\usepackage{enumerate}
\usepackage{amsthm,amsfonts,amsxtra,amscd}
\usepackage[all]{xy}
\usepackage[english]{babel}
\usepackage{verbatim}
\usepackage{hyperref}
\theoremstyle{plain}
\newtheorem{lem}{Lemma}[section]
\newtheorem{teo}[lem]{Theorem}
\newtheorem{prp}[lem]{Proposition}
\newtheorem{cor}[lem]{Corollary}

\newtheorem{claim}{Claim}

\theoremstyle{definition}
\newtheorem{dfn}[lem]{Definition}
\newtheorem{oss}[lem]{Remark}

\theoremstyle{remark}

\newcommand{\mQ}{\mathbb Q} \newcommand{\mR}{\mathbb R}

\newcommand{\mZ}{\mathbb Z}

\newcommand{\calE}{\mathcal E}  
 \newcommand{\calI}{\mathcal I} 
  
 \newcommand{\calO}{\mathcal O} 
  
  \newcommand{\calV}{\mathcal V}
 \newcommand{\calX}{\mathcal X}

\newcommand{\gob}{\mathfrak b}  
  \newcommand{\gog}{\mathfrak g}

  \newcommand{\gop}{\mathfrak p}
  
\newcommand{\got}{\mathfrak t} \newcommand{\gou}{\mathfrak u}

\newcommand{\sfA}  {\mathsf A}
\newcommand{\sfB}{\mathsf B} \newcommand{\sfC}{\mathsf C} \newcommand{\sfD}{\mathsf D}
\newcommand{\sfE}{\mathsf E}  \newcommand{\sfG}{\mathsf G}

\newcommand{\sfk}{\mathsf k}

\newcommand{\mru}{\mathrm u}

\newcommand{\gra}{\alpha} \newcommand{\grb}{\beta}       \newcommand{\grg}{\gamma}
\newcommand{\grd}{\delta} \newcommand{\gre}{\varepsilon} 
 \newcommand{\grl}{\lambda}     \newcommand{\grs}{\sigma}
\newcommand{\gro}{\omega}      
\newcommand{\grG}{\Gamma} \newcommand{\grD}{\Delta}
\newcommand{\grL}{\Lambda}

\newcommand{\mk}  {\Bbbk}%vuole amssymb

\newcommand{\ol}         {\overline}

\newcommand{\id} 			{\text{id}}
\newcommand{\ra}       {\rightarrow}
\newcommand{\lra}      {\longrightarrow}
\newcommand{\vuoto}    {\varnothing}

\newcommand{\isocan}   {\simeq}
\renewcommand{\geq}    {\geqslant}%vuole amssymb
\renewcommand{\leq}    {\leqslant}%vuole amssymb
\newcommand{\senza}    {\smallsetminus}

\DeclareMathOperator{\Ad}{Ad}
\DeclareMathOperator{\card} {card}
\DeclareMathOperator{\car}{char}
\DeclareMathOperator{\height} {ht}
\DeclareMathOperator{\Stab} {Stab}
\DeclareMathOperator{\supp} {supp}

            \newcommand{\st}       {\, : \,}
       
         \newcommand{\mand}     {\text{ and }}
        \newcommand{\mif}      {\text{ if }}

\title[The Bruhat order on Hermitian symmetric varieties]{The Bruhat order on Hermitian symmetric varieties\\
and on abelian nilradicals}

\author{Jacopo Gandini, Andrea Maffei}

\date{\today}

\email{jacopo.gandini@unipi.it}

\curraddr{\textsc{Dipartimento di Matematica\\ Universit\`a di Pisa\\ Largo Bruno Pontecorvo 5\\ 56127 Pisa, Italy}}

\email{maffei@dm.unipi.it}
  
\curraddr{\textsc{Dipartimento di Matematica\\ Universit\`a di Pisa\\ Largo Bruno Pontecorvo 5\\ 56127 Pisa, Italy}}

\begin{document}

\begin{abstract}
Let $G$ be a simple algebraic group and $P$ a parabolic subgroup of $G$ with abelian unipotent radical $P^\mru$, and let $B$ be a Borel subgroup of $G$ contained in $P$. Let $\gop^\mru$ be the Lie algebra of $P^\mru$ and $L$ a Levi factor of $P$, then $L$ is a Hermitian symmetric subgroup of $G$ and $B$ acts with finitely many orbits both on $\gop^\mru$ and on $G/L$. In this paper we study the Bruhat order of the $B$-orbits in $\gop^\mru$ and in $G/L$, proving respectively a conjecture of Panyushev and a conjecture of Richardson and Ryan.
\end{abstract} 

\maketitle
 
%%%%%%%%%%%%%%%%%%%%%%%%%%%%%%%%%%%%%%%
%%%%%%%%%%%%%%%%%%%%%%%%%%%%%%%%%%%%%%%
\section{Introduction}
%%%%%%%%%%%%%%%%%%%%%%%%%%%%%%%%%%%%%%%
%%%%%%%%%%%%%%%%%%%%%%%%%%%%%%%%%%%%%%%

Let $G$ be an almost simple algebraic group over an algebraically closed field $\mk$ of characteristic different from $2$. 
Let $P \subset G$ be a parabolic subgroup with abelian unipotent radical $P^\mru$ and let $P = L P^\mru$ be a Levi decomposition. 
Then the Levi subgroup $L$ is the identity component of the set of fixed points of an algebraic involution of $G$ if and only if $P^\mru$ is abelian,
in which case the homogeneous space $G/L$ is called a \textit{Hermitian symmetric variety}.

Let $B$ be a Borel subgroup of $G$ contained in $P$. Then $B$ acts with finitely many orbits on $G/L$ and on $\gop^\mru$, the Lie algebra of $P^\mru$. 
The aim of this paper is to give a combinatorial characterization of the corresponding Bruhat orders 
(that is, the partial order among $B$-orbits defined by the inclusion of orbit closures) proving a conjecture of Richardson 
and Ryan in the first case (see \cite{RS2}), and a conjecture of Panyushev in the second case (see \cite{Pa}). 

Fix a maximal torus $T$ in $B \cap L$ and let $\Phi$ be the root system of $G$ associated to $T$.
We denote by $\Delta \subset \Phi^+$ the set of the simple and the set of the positive roots determined by $B$, by $\Phi^-$ the negative roots, by $W$ the Weyl group of $\Phi$. Moreover, we denote by $s_\grb$ the reflection defined by a root $\grb$ and by $\ell$ the length of an element of $W$ determined by the choice of $\Delta$.

The Bruhat order of the $B$-orbits in $G/B$ was determined by Chevalley.  In this case $G=\bigsqcup_{w\in W} BwB$ and $\overline{BuB}\supset BvB$ if and only 
if $u\geq v$ with respect to the Bruhat order of $W$ (that is, the partial order generated by the relations $w s_\grb >w$ for all root $\grb$ such that 
$\ell(w s_\grb)>\ell(w)$). More generally if 
$Q\supset B$ is a parabolic subgroup of $G$, we have a similar description of the Bruhat order
of the $B$-orbits in $G/Q$. Let $W_Q\subset W$ be the Weyl group of $Q$, and let $W^Q$ be 
the set of minimal length representatives of the cosets in $W/W^Q$. Then 
$G=\bigsqcup_{w\in W^Q} BwQ$, and for $u,v\in W^Q$ we have $\overline{BuQ}\supset BvQ$ if and only if
$u\geq v$. If the unipotent radical of $Q$ is abelian, then the Bruhat order of $W^Q$ is particularly simple (see Proposition \ref{prp:GP1}). 

Let $\gop^\mru$ be the Lie algebra of $P^\mru$. The $B$-orbits in $\gop^\mru$, and more generally in any abelian ideal of the Lie algebra of $B$, were parametrized by Panyushev \cite[Theorem 2.2]{Pa} (see also Corollary \ref{cor:parametrizzazione} i)). Let $\Psi$ be the set of roots of $\gop^\mru$, and fix a root vector $e_\gra$ of weight $\gra$ for all $\gra \in \Psi$. If $S\subset \Psi$ is an orthogonal subset (that is, a subset of pairwise orthogonal roots), set $e_S=\sum_{\gra\in S} e_\gra$. Then the $B$-orbits
in $\gop^\mru$ are all of the form $B e_S$ for some orthogonal subset $S \subset \Psi$, and all such subsets give rise to distinct $B$-orbits. Since the action of $P^\mru$ on $\gop^\mru$ is trivial, denoting $B_L=B\cap L$, notice the $B$-orbits on $\gop^\mru$ coincide with the $B_L$-orbits on $\gop^\mru$.

The $B$-orbits in a general symmetric variety were studied by Richardson and Springer \cite{Sp,RS1,RS2,RS3}. There they proved that many similarities with the case of flag varieties hold, however a parametrization of the $B$-orbits 
in this setting is not so explicit as in the previous case. In the case of the Hermitian symmetric variety $G/L$ the parametrization
is much simpler and explicit, and was given in \cite[Theorem 5.2.4]{RS2} (see also Corollary \ref{cor:parametrizzazione} ii)). We describe it by using the language introduced above.

\begin{dfn}
Given $v \in W^P$ and an orthogonal subset $S \subset \Psi$, we say that the pair $(v,S)$ is \textit{admissible} if $v(S) \subset \Phi^-$. We denote by $V_L$ the set of the admissible pairs.
\end{dfn}

If $S \subset \Psi$ is an orthogonal subset, define a point in $G/L$ by setting $x_S = \exp(e_S)L$ (the definition of the exponential map for these particular elements is possible also in positive characteristic). Then to any admissible pair $(v,S)$ we associate the orbit $Bvx_S$: these orbits are all distinct, and every $B$-orbit in $G/L$ is of this form. Thus the $B$-orbits in $G/L$ are parametrized by the admissible pairs.

The link between the two parametrization is easy to explain. Consider the projection $\pi:G/L\lra G/P$ and let $w^P$ be the longest element in 
$W^P$, it satisfies $w^P(\Psi)\subset \Phi^-$. The stabilizer of $w^PP$ inside $B$ is equal to $B_L$ and the 
fiber of $\pi$ over $w^PP$ is isomorphic to $\gop^\mru$. Hence the $B_L$-orbits in $\gop^\mru$ correspond exactly to the $B$-orbits in $Bw^PP$. 

In order to study the Bruhat order on a symmetric variety (still denoted by $\leq$), the approach of Richardson and Springer 
is to look at the action of the minimal parabolic subgroups ($\supsetneq B$) of $G$ on the set of the $B$-orbits. If $\gra \in \Delta$, let 
$P_\gra \subset G$ be the associated minimal parabolic subgroup containing $B$. If $\calO$ is a $B$-orbit in $G/L$ (or more generally in any symmetric variety), then $P_\gra \calO$ decomposes in the union of at most three $B$-orbits. Let $m_\gra \!\cdot \calO$ be the open $B$-orbit in $P_\gra \calO$, then obviously $\calO \leq m_\gra \cdot \calO$. As already in the case of flag varieties, the Bruhat order it is not generated by the relations $\calO\leq m_\gra \!\cdot \calO$, however it is possible to reconstruct it from 
the action of the minimal parabolic subgroups in the following way. 
Given $\gra_0,\gra_1,\dots,\gra_m\in \Delta$ and a $B$-orbit $\calO \subset G/L$, set $\calO_1=m_{\gra_m}\!\cdots m_{\gra_1} \!\cdot \calO$ and 
$\calO_2=m_{\gra_m}\!\cdots m_{\gra_1} \!\cdot m_{\gra_0}\!\cdot \calO$: then $\calO_1$ is contained in the closure of $\calO_2$, and the Bruhat order on $G/L$
is generated 
by these kind of relations (see \cite[Theorem 7.11]{RS1}). We will not directly use this result, but we will make use of some of 
its consequences, and more generally we will make use of the action of the minimal parabolic subgroups.

Using the action of the minimal parabolic subgroups it is also not difficult to give a formula for the dimension of a $B$-orbit in a symmetric variety (see \cite[Lemma 7.2]{RS1}). 
In our case the formula can be expressed as follows. If $S$ is a set of orthogonal roots, denote $\grs_S = \prod_{\grb\in S} s_\grb$.
If $(v,S) \in V_L$, we have then
\begin{equation}\label{eq:dim1}
 \dim Bvx_S = \card \Psi +\frac {\ell(\grs_{v(S)})+\card S}{2}.
\end{equation}
From this formula it's also easy to deduce a dimension formula for the $B$-orbits in $\gop^\mru$ conjectured by Panyushev, see Corollary 
\ref{cor:dim-formula}. 

We now come to the main results of the paper. We first describe the Bruhat order on $\gop^\mru$. Let $w_P$ be the longest element in $W_P$. 

\begin{teo} [Corollary \ref{cor:panyushev-conj}] \label{teo:pan}
Suppose that $S,S' \subset \Psi$ are orthogonal subsets, then $Be_S \subset \ol{B e_{S'}}$ if and only if $\grs_{w_P(S)} \leq \grs_{w_P(S')}$.
\end{teo}

The previous theorem was conjectured by Panyushev (see \cite[Conjecture 6.2]{Pa}).
When $G$ is of type $\sfA$ or $\sfC$, it can be deduced from more general results on the adjoint and coadjoint $B$-orbits of 
nilpotency order 2 and their Bruhat order, studied by Ignatyev \cite{Ign}, Melnikov \cite{Mel}, and Barnea and Melnikov \cite{BM1}. 
For orthogonal groups this formula was proved by Barnea and Melnikov \cite{BM2}, so only the exceptional cases remained to be proved. However our proof is not based on a case-by-case analysis, and it does not rely on such results.

We now come to the Bruhat order on $G/L$. Beyond $v$ and $\grs_{v(S)}$, there is a third Weyl group element that one can canonically attach to a $B$-orbit in $G/L$. 
If indeed $P^-$ denotes the opposite parabolic subgroup of $P$, then $L= P \cap P^-$. Thus one can define an element 
$\nu \in W^P$ by the equality $Bv x_S P^- = B \nu P^-$. The element $\nu$ can be easily described in terms of $(v,S)$: if indeed $[w]^P \in W^P$ denotes the minimal length representative of the coset $wW_P$, then we 
have $\nu = [v\grs_S]^P$, see Lemma \ref{lem:twisted-involution}.

Define a partial order on the set of the admissible pairs $V_L$ as follows:
\begin{equation}\label{eq:ordVL}
(u,R) \leq (v,S) \;\quad \text{ if } \quad \;
	[v\grs_S]^P \leq  [u\grs_R]^P \leq u \leq v
	\; \text{ and } \;
	\grs_{u(R)} \leq \grs_{v(S)}.
\end{equation}
Then we will prove the following theorem, which was conjectured by Richardson and Ryan (see \cite[Conjecture 5.6.2]{RS2}).

\begin{teo}[Theorem \ref{teo:conRR}] \label{teo:RS2}
Let $(u,R), (v,S)$ be admissible pairs. Then $Bux_R \subset \ol{Bvx_S}$ if and only if $(u,R) \leq (v,S)$.
\end{teo}

If an orbit is in the closure of another one, the fact that the above combinatorial conditions have to hold was known. This is easily proved in the case of $\gop^\mru$ (see the first paragraph of the proof of Theorem \ref{teo:panv}), 
and it follows from the work of Richardson and Springer in the other case. So, what we really need to prove is the other implication. 

Richardson and Ryan proved some partial results in this direction which were reported in \cite{RS2}. When $G$ is of type $\sfA$, the $B$-orbits in $G/L$ have also been studied by Matsuki-Oshima \cite{MO} and Yamamoto \cite{Ya} 
in terms of suitable combinatorial parameters called \textit{clans}. In this case, the Bruhat order on $G/L$ was recently determined by Wyser \cite{Wy}.

By identifying the $B$-orbits in $\gop^\mru$ with the $B$-orbits in $B w^P P$ as explained above, it is not hard to see that the Bruhat 
order on $\gop^\mru$ appears as a particular case of the Bruhat order on $G/L$. However we will need
to study first this particular case, and more precisely to determine the Bruhat order of the $B$-orbits in  
the subsets $BvP/P$ for a fixed $v \in W^P$. We will prove in this case an analogue of Theorem 
\ref{teo:pan} (see Theorem \ref{teo:panv}), which will be used as the basis of an induction to prove Theorem \ref{teo:RS2}.

The abelian nilradicals $\gop^\mru$ are special instances of abelian ideals of $\gob$, and actually the mentioned parametrization of 
the $B$-orbits in $\gop^\mru$ in  \cite[Theorem 2.2]{Pa} holds in this more general setting without substantial differences. 
We will come back to this point in a forthcoming article \cite{GMMP}, where we will study the Bruhat order on arbitrary abelian ideals of $\gob$.

The paper is organized as follows. In Section \ref{sez:GP} we prove some preliminary results concerning the Bruhat order
on $G/P$. In Section \ref{sez:inv} we recall some results from \cite{RS1,RS2} about the Bruhat order on the set of involutions. 
We also prove here some additional results which apply in the Hermitian case that we will need later. In Section \ref{sez:par}
we describe the parametrization of the $B$-orbits in $\gop^\mru$ and in $G/L$. 
In Section \ref{sez:azioneminimali} we describe some results about the action of the minimal parabolic subgroups, and prove the dimension formulas.
In Section \ref{sez:Pu} we prove Theorem \ref{teo:pan}, whereas Section \ref{sez:GL1} is devoted to the prooof of Theorem \ref{teo:RS2}.

\textit{Acknowledgments.} We thank the anonymous referees for their careful readings and for several useful remarks and suggestions.

\subsection{Notation and preliminaries}
We keep the notation used in the Introduction. Moreover, we will make use of the following conventions.

If $H$ is any algebraic group, we will denote its unipotent radical by $H^\mru$ and its character lattice by $\calX(H)$, and we set $U = B^\mru$. The Lie algebra of $H$ will be denoted by the corresponding fraktur letter, or when this creates some ambiguity by $\mathrm{Lie}(H)$. Notice that the Lie algebra of $G$ depends on the isogeny class of $G$, however the only Lie algebra we will be interested is $\gop^\mru$, which is independent of the isogeny class of $G$.

We denote by $\grL$ the weight lattice of $T$, and regard the vector space $\grL \otimes_\mZ \mQ$ as a Euclidean space with a $W$-invariant nondegenerate scalar product. As usual, the monoid of the dominant weights defined by $B$ is denoted by $\grL^+$.

If $v\in W$, then we define
$$ \Phi^+(v)=\{\gra \in \Phi^+\st v(\gra)\in \Phi^-\}. $$
If $\gra \in \Phi$, the corresponding coroot will be denoted by $\gra^\vee$.
If moreover $\gra \in \grD$, then the corresponding fundamental weight (resp. coweight) will be denoted by $\gro_\gra$ (resp. by $\gro_\gra^\vee$). 
We will denote by $\theta$ the highest root of $\Phi$.

If $\gra \in \grD$ and $\grb \in \Phi$, we denote by $[\grb: \gra]$ the coefficient of $\gra$ in $\grb$. The \textit{height} of $\grb$ is defined 
by $\height(\grb) = \sum_{\gra \in \grD} [\grb:\gra]$. 
We will regard the weight lattice as a partially ordered set with the \textit{dominance order}: if 
$\grl, \mu \in \grL$, then we write $\grl \leq \mu$ if $\mu - \grl$ is a sum of simple roots. 
Similarly, we will regard  the coweight lattice as a partially ordered set with the dominance order defined by the simple coroots.

We say that a nonzero dominant weight $\mu$ is \textit{minuscule} if it is minimal in $\grL^+$. 
Similarly, we say that a dominant coweight $\mu$ is \textit{cominuscule} if it is a minuscule weight for the root system 
$\Phi^\vee$ of the coroots. 
Recall the following characterizations of cominuscule elements (see \cite[VIII, \S 7, no. 3]{Bou} and \cite[Exercise 13.13]{Hu}).

\begin{prp}\label{prp:minuscoli1}
Let $\mu$ be a dominant coweight, then the following conditions are equivalent:
\begin{enumerate}[\indent i)]
	\item $\mu$ is cominuscule;
	\item $\langle \mu, \gra \rangle \leq 1$ for all $\gra \in \Phi^+$;
	\item $\mu = \gro^\vee_\gra$, for some $\gra \in \grD$ such that $[\theta:\gra]=1$.
\end{enumerate}
\end{prp}

If $\gra \in \Phi$, we will denote by $\gou_\gra \subset \gog$ the corresponding root space, and by $U_\gra \subset G$ and $u_\gra(t) : \mk \ra U_\gra$ respectively the corresponding 
root subgroup and a one parameter subgroup. We can choose the one parameter subgroup so that 
$s_\gra=u_\gra(t)\,u_{-\gra}(-t^{-1})\,u_\gra(t)\, T$ for all $t \in \mk^\times$ (see \cite{Springer} Lemma 8.1.4 i)).
If $w\in W$, here and throughout the paper by abuse of notation we will denote by the same letter any representative of $w$ in the normalizer of $T$ in $G$. 

Fix root vectors $e_\gra \in \gou_\gra$ and $f_\gra\in \gou_{-\gra}$ for all $\gra \in \Phi^+$. If $S \subset \Psi$ is an orthogonal subset, recall the element $ e_S = \sum_{\gra \in S} e_\gra$ defined in the Introduction, and define similarly
$f_S = \sum_{\gra \in S} f_\gra$. If $e \in \gou$, write $e = \sum_{\gra \in \Phi^+} c_\gra e_\gra$ and define the \textit{support} of $e$ as 
$$ \supp(e) = \{ \gra \in \Phi^+ \st c_\gra \neq 0\}.$$

Since $\gop^\mru$ is abelian, it is well defined a $P$-equivariant exponential map 
$$
\exp:\gop^\mru\lra P^\mru
$$
in all characteristics (see \cite[Proposition 5.3]{Se}). This map is an isomorphism of varieties, and it satisfies
$\exp(x+y)=\exp(x)\cdot\exp(y)$. A similar map $\exp:(\gop^-)^\mru\ra (P^-)^\mru$, still denoted by the same symbol, exists for $P^-$ as well. If $\gra\in \Psi$, when convenient we will choose the root vectors $e_\gra$ and $f_\gra$ so that $u_\gra(t)=\exp(t\,e_\gra)$ and $u_{-\gra}(t)=\exp(t\,f_\gra)$, for all $t\in \mk$. If $S\subset \Psi$ is an orthogonal subset and if $\gra, \grb \in S$ are not equal, then the elements $u_\gra$ and $u_{-\grb}$ commute (see Lemma \ref{lem:panyushev} i)), thus we have
$$
\grs_S=\exp(t\, f_S)\exp(-t^{-1}\,e_S)\exp(t\, f_S)T.
$$

Finally we make some remarks about the characteristic of $\mk$. Notice that parabolic subgroups with abelian unipotent radical only occur when $G$ is a classical group, or when it is of type $\sfE_6$ or $\sfE_7$ (see e.g. \cite[Remark 5.1.3]{RS2}). Since $\car \mk \neq 2$ and since $G$ is never of type $\sfG_2$, the following property holds: let $\gra,\grb \in \Phi$ be such that $\gra \neq \pm \grb$ and let $m$ be maximal such that
$\gra+m\grb$ is a root (thus $m \leq 2$ by our assumptions), then for all $t \in \sfk^\times$ we have
$$
 u_\grb(t) \cdot  e_\gra = e_\gra + c_1\, t\, e_{\gra + \beta}+\dots+c_m\, t^m\,e_{\gra+m\grb}
$$
for some nonzero constants $c_1,\dots,c_m$ (see for example the construction of Chevalley groups in \cite{Steinberg} Sections 1,2,3). 
If we assume that $G$ is simply laced (in which case $m\leq 1$), then the same property holds in characteristic $2$ as well: in this case, 
we expect that Theorems \ref{teo:pan} and \ref{teo:RS2} still hold (notice that in the articles of Richardson and Springer it is 
always required $\car \mk \neq 2$). 
On the other hand, Theorems \ref{teo:pan} and \ref{teo:RS2}, and even the parametrizations of Corollary \ref{cor:parametrizzazione},
are false if $G = \mathrm{Sp}_4$ and $\car \mk = 2$, see Subsection \ref{ex:Sp4}.

%%%%%%%%%%%%%%%%%%%%%%%%%%%%%%%%%%%%%%%
%%%%%%%%%%%%%%%%%%%%%%%%%%%%%%%%%%%%%%%
\section{Some remarks on the Bruhat order on $G/P$}\label{sez:GP}
%%%%%%%%%%%%%%%%%%%%%%%%%%%%%%%%%%%%%%%
%%%%%%%%%%%%%%%%%%%%%%%%%%%%%%%%%%%%%%%

%%%%%  \begin{oss}	\label{oss:cominuscule}
%%%%%  Notice that $\mu$ is cominuscule if and only if $\langle \mu, \gra^\vee \rangle = \tfrac{||\theta||^2}{||\gra||^2}$ for all $\gra \in \Phi^+$.
%%%%%  \end{oss}

We will freely make use of standard properties of the Bruhat order on $W$ and on $W^P$ (see e.g. \cite{BB}).
In particular recall  that, if $u,v\in W$ and $u\geq v$, then $[u]^P\geq [v]^P$ as well, and that 
if $v\in W^P$ and $s_\gra v<v$ for some $\gra\in\Delta$, then $s_\gra v\in W^P$ as well.

In this section we will prove a characterization of the Bruhat order on $W^P$ in case $P$ is a parabolic subgroup of $G$ with abelian unipotent radical, as in our assumptions. 
In this case $P$ is a maximal parabolic subgroup, corresponding to a simple root $\gra_P$ such that $[\theta: \gra_P] = 1$.
We denote by $\gro_P$ and $\gro_P^\vee$ respectively the fundamental weight 
and the fundamental coweight defined by $\gra_P$, thus $\gro_P^\vee$ is a cominuscule coweight by Proposition \ref{prp:minuscoli1} i). Denote also $\Delta_P=\Delta\senza\{\gra_P\}$, let
$\Phi_P \subset \Phi$ be the root system generated by $\Delta_P$, and set $\Phi^+_P =\Phi^+\cap\Phi_P$.

Recall that $\Psi$ is the set of $T$-weights of $\gop^\mru$. Since $[\theta : \gra_P] = 1$, we have
\begin{equation}	\label{eqn:Psi}
\Psi = \{\gra \in \Phi^+ \st [\gra : \gra_P] = 1\}.
\end{equation}
If $v \in W$, notice that $v \in W^P$ if and only if $v(\Phi^+_P) \subset \Phi^+$, namely $v \in W^P$ if and only if $\Phi^+(v) \subset \Psi$.

\begin{lem}\label{lem:GP1}
Let $v \in W^P$ and let $\gra \in \Phi^+$ be such that $s_\gra v\in W^P$ and $\ell(s_\gra v) = \ell(v) -1$, then $\gra \in \grD$.
\end{lem}

\begin{proof}
Denote $u = s_\gra v$, by \cite[Proposition 3.1.3]{BB} it is enough to show that $\Phi^+(u) \subset \Phi^+(v)$. Denote $\grb = -v^{-1}(\gra)$ and let $\grg \in \Phi^+(u)$, then $v(\grg) = us_\grb(\grg) = u(\grg) - \langle \grg, \grb^\vee \rangle \gra$. Notice that $\grb,\grg \in \Psi$ because $u, v \in W^P$. Thus $\grb+\grg \not \in \Phi$ because $\gop^\mru$ is abelian. Therefore $\langle \grg, \grb^\vee \rangle \geq 0$, and $v(\grg) \in \Phi^-$.
\end{proof}

\begin{lem}	\label{lem:altezza}
Let $w \in W^P$, then $\ell(w) = \height(\gro_P^\vee - w \gro_P^\vee)$.
\end{lem}

\begin{proof}
Let $w = s_n \cdots s_1$ be a reduced expression and, for $i \leq n$, denote $w_i = s_i \cdots s_1$. Let $\gra_i \in \grD$ 
be the simple root corresponding to $s_i$ and denote $\grb_i = w_{i-1}^{-1}(\gra_i)$. Notice that $\grb_i \in \Psi$ since 
$w_i \in W^P$, so that $[\grb_i : \gra_P] = 1$. Hence we have
$$w_i(\gro_P^\vee) = w_{i-1}(\gro_P^\vee) - \langle \gro_P^\vee, \grb_i \rangle \gra_i = w_{i-1}(\gro_P^\vee) - \gra_i. $$
Therefore $\height(\gro_P^\vee - w_i(\gro_P^\vee)) = \height(\gro_P^\vee - w_{i-1}(\gro_P^\vee)) +1$, and the claim follows.
\end{proof}

The equivalence of i) and ii) in the following proposition was already known (see \cite[Theorem 7.1]{St1} and \cite[Corollary 3.12]{EHP}), we thank F.~Brenti 
for pointing out the reference.

\begin{prp}\label{prp:GP1}
 Let $u,v\in W^P$, then the following conditions are equivalent:
\begin{enumerate}[\indent i)]
 \item $u \leq v$;
 \item $\Phi^+(u)\subset\Phi^+(v)$;
 \item $v(\omega_P^\vee)\leq u(\omega_P^\vee)$;
 \item $v(\omega_P)\leq u(\omega_P)$.
 \end{enumerate}
\end{prp}

\begin{proof}
i) $\Rightarrow$ ii). Let $n= \ell(u) - \ell(v)$, by the chain property (see \cite[Theorem 2.5.5]{BB}) there exists a chain $u = u_0 < u_1 < \ldots < u_n = v$ of elements in $W^P$ such that $\ell(u_i) = \ell(u_{i-1}) + 1$ for all $i \leq n$. Therefore by Lemma \ref{lem:GP1} we get $\Phi^+(u_{i-1}) \subset \Phi^+(u_i)$ for all $i \leq n$, hence $\Phi^+(u)\subset\Phi^+(v)$.

ii) $\Rightarrow$ i). This is well known, and it holds for all $u,v \in W$ (see \cite[Proposition 3.1.3]{BB}).

i) $\Rightarrow$ iii). By considering any chain in $W$ between $u$ and $v$, we can assume that $\ell(u) = \ell(v) +1$. 
Let $\gra \in \Phi^+$ be such that $v = s_\gra u$, and denote $\grb = u^{-1}(\gra) $. Then $\grb \in \Phi^+$, and since 
$\gro_P^\vee$ is dominant it follows that $\langle u( \gro_P^\vee) , \gra \rangle = \langle \gro_P^\vee, \grb \rangle \geq 0$. Therefore
$$v(\gro_P^\vee) = s_\gra u(\gro_P^\vee) = u (\gro_P^\vee) - \langle u( \gro_P^\vee) , \gra \rangle \gra^\vee \leq u (\gro_P^\vee).$$

iii) $\Rightarrow$ i).
We show that $u \leq v$ proceeding by induction on $h=\height(u(\gro_P^\vee) - v (\gro_P^\vee))$. By Lemma \ref{lem:altezza} we have 
$\ell(v) - \ell(u) = \height(u(\gro_P^\vee) - v (\gro_P^\vee))$.

Suppose that $h = 0$. Then $u(\gro_P^\vee) = v (\gro_P^\vee)$. On the other hand the orbit map $W^P \ra W \gro_P^\vee$ is bijective because $W_P$ is the stabilizer of $\gro_P^\vee$, therefore it must be $u=v$.

Suppose now that $h > 0$, and write $u (\gro_P^\vee) = v (\gro_P^\vee) + \gra_1^\vee + \ldots + \gra_h^\vee$ with $\gra_i \in \grD$. 
Denote $\grb_i = u^{-1}(\gra_i)$, then $\gro_P^\vee = u^{-1}v (\gro_P^\vee) + \grb_1^\vee + \ldots + \grb_h^\vee$. 
Notice that for all $w \in W \senza W_P$ it holds $w(\gro_P^\vee)\leq \gro_P^\vee - \gra_P^\vee$. Since by definition 
$u^{-1} v \not \in W_P$, it follows that $\gra_P \leq \grb_1 + \ldots + \grb_h$. Thus at least one of the $\grb_i$ must 
be a positive root supported on $\gra_P$ (namely $\grb_i \in \Psi$). It follows that $u < s_{\gra_i} u$ and $[s_{\gra_i}u]^P \neq u$ which implies
$[s_{\gra_i}u]^P > u$. Since $\ell(s_{\gra_i}u)=\ell(u)+1$ we deduce that $s_{\gra_i}u\in W^P$. 
On the other hand by construction we have
$$s_{\gra_i} u(\gro_P^\vee) = u(\gro_P^\vee) - \langle \gro_P^\vee, \grb_i \rangle \gra_i^\vee = u(\gro_P^\vee) - \gra_i^\vee,
$$
hence $v(\gro_P^\vee) \leq s_{\gra_i}u(\gro_P^\vee)$. By the inductive hypothesis we get then $s_{\gra_i}u \leq v$, thus $u < v$.

iii) $\Leftrightarrow$ iv). This is obvious.
\end{proof}

Given $v\in W^P$, we now describe the roots $\gra \in \grD$ such that $s_\gra v <v$ by making use of the poset $\Phi^+(v)$. If $\gra$ is such a simple root, notice that $\ell(s_\gra v) = \ell(v) -1$, and that $\Phi^+(v)=\Phi^+(s_\gra v)\sqcup \{- v^{-1}(\gra)\}$.

\begin{lem}	\label{lem:dominanza}
Let $\gra, \grb \in \Phi^+$ and suppose that $\gra \leq \grb$. Then there exist $\gra_1, \ldots, \gra_n \in \grD$ such that $\gra + \gra_1 + \ldots + \gra_i \in \Phi^+$ for all $i \leq n$, and $\grb = \gra + \gra_1 + \ldots + \gra_n$.
\end{lem}

\begin{proof}
Let $\grb - \gra = \gra_1 + \ldots + \gra_n$ be any expression of $\grb-\gra$ as a sum of simple roots, we show the claim by induction on $n$. 
The claim is obvious if $n = 1$, so we can assume that $n > 1$. Notice that for some $i \leq n$ it must be either $( \gra, \gra_i) < 0$ or 
$(\grb, \gra_i)> 0$: otherwise $||\grb-\gra||^2 = \sum_{i=1}^n (\grb-\gra, \gra_i) \leq 0$, hence $\gra = \grb$, a contradiction. For such 
a choice of $\gra_i$, we have either that $\gra + \gra_i \in \Phi^+$ and $\gra < \gra + \gra_i < \grb$, or that $\grb - \gra_i \in \Phi^+$ and 
$\gra < \grb - \gra_i < \grb$. Therefore the claim follows applying the induction.
%\footnote{A: forse in una versione da spedire ad una rivista  possiamo dire anche solo ben noto. J: io sarei pi\'u per lasciarlo, ma si pu\'o tranquillamente togliere}
\end{proof}

\begin{prp}\label{prp:GP2}
Let $v\in W^P$.
\begin{enumerate}[\indent i)]
	\item Let $\gra\in \Delta$ be such that $s_\gra v<v$ and denote $\grb = -v^{-1}(\gra)$. Then $\grb$ is maximal in $\Phi^+(v)$, and minimal in $\Psi \senza \Phi^+(s_\gra v)$.
	\item Let $\grb \in \Phi^+(v)$ be a maximal element and denote $\gra = -v(\grb)$. Then $\gra \in \Delta$ and $s_\gra v<v$.
	\item  Let $\grb \in \Psi \senza \Phi^+(v)$ be a minimal element and denote $\gra = v(\grb)$. Then $\gra \in \Delta$ and $s_\gra v>v$.
\end{enumerate}
\end{prp}

\begin{proof}
i).  We prove the first claim, the second one is proved similarly. Suppose that $\grb$ is not maximal in $\Phi^+(v)$ and let $\grg \in \Phi^+(v)$ be such that $\grb < \grg$, we can assume that $\grg$ is minimal with this property. Since $\grg \neq \grb$, we have $\grg \in \Phi^+(s_\gra v)$. By Lemma \ref{lem:dominanza}, there exists $\gra' \in \grD$ such that $\grg-\gra' \in \Phi^+$ and $\grb \leq \grg - \gra'$. Since $\grb, \grg \in \Psi$, by \eqref{eqn:Psi} we have $[\grb : \gra_P] = [\grg : \gra_P] = 1$, thus $\gra' \in \grD_P$. Notice that it must be $\grb = \grg-\gra'$: otherwise the minimality of $\grg$ would imply that $v(\grg-\gra') > 0$, hence $v(\grg) > 0$  because $v(\grD_P) \subset \Phi^+$. Therefore we get a contradiction because $s_\gra v(\grg) = s_\gra v(\gra') + s_\gra v(\grb) > 0$.

ii). We only show that $\gra \in \grD$, as the last claim is obvious. Suppose that $\gra \not \in \grD$. Since $\gra \in \Phi^+$, by Lemma \ref{lem:dominanza} there exists $\gra' \in \grD$ such that $\gra - \gra' \in \Phi^+$. Denote $\grb' = -v^{-1}(\gra')$ and $\grg = \grb' - \grb$. Notice that $\grg \in \Phi$, indeed we have $\grg = v^{-1}(\gra - \gra')$.

Suppose that $\grg \in \Phi^+$. Then $\grb' \in \Phi^+(v)$ and $\grb < \grb'$, contradicting the maximality of $\grb$. Therefore it must be $\grg \in \Phi^-$. Denote $\grg' = - \grg$, then $\grg' \in \Phi^+(v)$. Since $\grg'$ is comparable with $\grb$, by the maximality of $\grb$ we get $\grg' < \grb$, namely $\grb' \in \Phi^+$. Thus by i) it follows that $\grb'$ is also maximal in $\Phi^+(v)$, and since $\grb' < \grb$ we get a contradiction.

iii). This is proved in a similar way to ii).
\end{proof}

\begin{oss}
Notice that the statements of Proposition \ref{prp:GP2} are false if $P^\mru$ is not abelian. 
Suppose indeed that $\Phi$ is of type $\sfB_2$ and that $P$ is the maximal parabolic subgroup 
defined by $\gra_P = \gra_2$. Let $v = w^P$ be the longest element in $W^P$. Then we have 
$v = s_2 s_1 s_2$ and $s_2 v < v$. However $\Phi^+(v) = \{\gra_2, \gra_1+\gra_2, \gra_1 + 2\gra_2\}$ and $-v^{-1}(\gra_2) = \gra_1 + \gra_2$. 
\end{oss}

Finally, we recall some general well-known properties of the Bruhat order on $W$, that will be used repeatedly. The last property is usually referred to as the \textit{lifting property}.

\begin{lem}\label{lem:parallelogramma}
Let $u,v \in W$ with $u < v$ and let $\gra \in \grD$.
\begin{enumerate}[\indent i)]
 	\item Suppose that $s_\gra u > u$ and $s_\gra v > v$, then $s_\gra u < s_\gra v$.
	\item Suppose that $s_\gra u < u$ and $s_\gra v < v$, then $s_\gra u < s_\gra v$.
	\item Suppose that $s_\gra u > u$ and $s_\gra v < v$, then $u < s_\gra v$ and $s_\gra u < v$.
\end{enumerate}
Similar statements hold if we consider the multiplication by $s_\gra$ on the right.
\end{lem}

\begin{proof}
Properties i) and ii) easily follow from the definition of the Bruhat order on $W$. For the last property, see \cite[Proposition 2.2.7]{BB}
\end{proof}

%%%%%%%%%%%%%%%%%%%%%%%%%%%%%%%%%%%%%%%
%%%%%%%%%%%%%%%%%%%%%%%%%%%%%%%%%%%%%%%
\section{The Bruhat order on the set of involutions} \label{sez:inv}
%%%%%%%%%%%%%%%%%%%%%%%%%%%%%%%%%%%%%%%
%%%%%%%%%%%%%%%%%%%%%%%%%%%%%%%%%%%%%%%

In this section we recall some results from \cite{RS1,RS2,RS3}. Let $\calI$ be the set of the involutions in $W$. If $\gra \in \Delta$ and $\grs \in \calI$, following Richardson and Springer \cite{RS1} we define
$$
s_\gra \circ \grs=
\begin{cases}
 s_\gra \grs &\mif s_\gra \grs = \grs s_\gra \\
 s_\gra \grs s_\gra &\mif s_\gra \grs \neq \grs s_\gra
\end{cases}
$$
Notice that $s_\gra \circ \grs = \tau$ if and only if $s_\gra \circ \tau = \grs$.

Replacing the ordinary action of the simple reflections with the circle action defined above, 
several well-known properties of the Bruhat order on $W$ carry over to the Bruhat order on $\calI$.
The results of the following two lemmas are contained in \cite{RS1}. There they are stated for what Richardson 
and Springer call the \textit{standard order} on $\calI$, which in \cite{RS3} is proved to be equivalent to the usual Bruhat order. Their proof 
follows from general results whose proofs are spread across the two papers, for this reason we prefer to give a direct proof here
since it is quite short. This proof works for any Coxeter group without the assumption of finiteness. 

\begin{lem} [{\cite[3.2]{RS1}}]\label{lem:inv0}
Let $\gra\in \Delta$ and $\grs \in \calI$. Then the following hold:
\begin{enumerate}[\indent i)]
 	\item  $s_\gra\circ \grs >\grs$ if and only if $s_\gra \grs > \grs$;
	\item $s_\gra\circ \grs < \grs $ if and only if $s_\gra\grs <\grs$.
\end{enumerate}
If moreover $s_\gra\grs \neq \grs s_\gra$ and $s_\gra\circ \grs >\grs$ (resp.\ $s_\gra\circ \grs <\grs$), 
then $s_\gra \grs s_\gra> s_\gra \grs> \grs$ and $s_\gra \grs s_\gra> \grs s_\gra > \grs$ 
(resp.\ $s_\gra \grs s_\gra< s_\gra \grs< \grs$ and $s_\gra \grs s_\gra<  s_\gra \grs< \grs$)
\end{lem}

\begin{proof}
If $s_\gra\grs = \grs s_\gra$ there is nothing to prove. Assume $s_\gra\grs \neq \grs s_\gra$. 
If $s_\gra \grs > \grs$, then $\grs(\gra)=\grs^{-1}(\gra)>0$, hence $\grs s_\gra  > \grs$. 
Notice that $\grs(\gra)\neq \gra$, otherwise $\grs s_\gra \grs =s_\gra$. Therefore $(\grs s_\gra)^{-1}(\gra)=(s_\gra\grs)(\gra)=
s_\gra(\grs(\gra))>0$, and we get $s_\gra\grs s_\gra >s_\gra\grs$,  $s_\gra\grs s_\gra >\grs s_\gra$.
The case $s_\gra \grs < \grs$ is similar. 
\end{proof}

\begin{lem} [{\cite[8.13 and 8.14]{RS1}}] \label{lem:inv1}
Let $\gra\in \Delta$ and $\grs, \tau \in \calI$, and suppose that $\grs < \tau$. Then the following hold:
\begin{enumerate}[\indent i)]
 \item  if $s_\gra \circ \grs > \grs$ and $s_\gra \circ \tau > \tau$, then $s_\gra \circ \grs < s_\gra \circ \tau$;
 \item  if $s_\gra \circ \grs < \grs$ and $s_\gra \circ \tau < \tau$, then $s_\gra \circ \grs < s_\gra \circ \tau$; 
 \item  if $s_\gra \circ \grs > \grs$ and $s_\gra \circ \tau < \tau$, then $s_\gra \circ \grs \leq \tau$ and $\grs \leq s_\gra \circ \tau$.
\end{enumerate}
\end{lem}
\begin{proof}
We prove iii), the proofs of i) and ii) are similar. 

If $s_\gra \grs = \grs s_\gra$ and $s_\gra \tau=\tau s_\gra$ there is nothing to prove. 
If $s_\gra \grs \neq \grs s_\gra$ and $s_\gra \tau=\tau s_\gra$, then by Lemma \ref{lem:parallelogramma} iii) applied to $v=\tau$ and 
$u=\grs$ we deduce that $\tau>  s_\gra \grs$ and $s_\gra \circ \tau \geq \grs$. Now if we apply the analogue respect to the right multiplication of Lemma \ref{lem:parallelogramma} iii) to the case $v=\tau$ and $u=s_\gra \grs$, we deduce that $\tau\geq s_\gra \circ \grs$. If $s_\gra \grs \neq \grs s_\gra$ and $s_\gra \tau=\tau s_\gra$ or if $s_\gra \grs \neq \grs s_\gra$ and $s_\gra \tau \neq \tau s_\gra$ the proof is similar.
\end{proof}

Following Richardson and Springer \cite[3.9]{RS1}, define the \textit{length} of $\grs \in \calI$ as an involution to be
$$
	L(\grs) =  \frac{\ell(\grs) + \grl(\grs)}{2},
$$
where $\grl(\grs)$ denotes the dimension of the $-1$ eigenspace of $\grs$ on $\grL \otimes_\mZ \mR$.
If $S = \{\grb_1, \ldots, \grb_n\}$ is a set of orthogonal roots and $\grs_S = s_{\grb_1} \cdots s_{\grb_n}$, then 
the eigenspace of $\grs_S$ on $\grL \otimes_\mZ \mR$ of eigenvalue $-1$ is generated by $S$, hence the formula above takes the form
$$
	L(\grs_S)=\frac{\ell(\grs_S) + \card S}{2}.
$$

%\begin{equation} \label{lem:inv3}
%L(\grs_S)=\tfrac 12 (\card(S)+\ell(\grs_S)).
 %\end{equation}

\begin{oss}
Notice that the previous definition agrees with that of \cite{RS1}, 
in case the involution of $G$ acts trivially on $W$ (as it happens in the Hermitian case).
\end{oss}

The length function $L$ is compatible with the circle action of the simple reflections and with the Bruhat order. The following is a consequence of Lemma \ref{lem:inv0}.

\begin{lem}[{\cite[3.18]{RS1}}] \label{lem:inv2}
Let $\gra \in \grD$ and $\grs \in \calI$. Then
$$
L(s_\gra \circ \grs) =
\begin{cases}
L(\grs)+1 &\mif s_\gra \circ \grs > \grs \\
L(\grs) -1&\mif s_\gra \circ \grs < \grs
\end{cases}
$$
\end{lem}

The following result is stated in \cite{RS1} for the standard order. Again, since the proof is not hard, we prefer to give a direct proof.

\begin{lem}[{\cite[8.1]{RS1}}]	\label{lem:bruhat-Lcompatibile}
Let $\grs, \tau \in \calI$ be such that $\grs \leq \tau$ and $L(\grs) \geq L(\tau)$. Then $\grs = \tau$.
\end{lem}

\begin{proof}
We proceed by induction on $\ell(\tau)$. If $\ell(\tau) = 0$, then $\grs=\tau=\id$.
Suppose that $\ell(\tau)>0$, let $\gra\in \Delta$ be such that $s_\gra\tau < \tau$ and denote $\tau' = s_\gra\circ \tau$. Then $\tau' <\tau$ by Lemma \ref{lem:inv0} ii), thus $L(\tau')=L(\tau)-1$ by Lemma \ref{lem:inv2}. Denote $\grs' = s_\gra\circ \grs$. 

Suppose that $\grs' >\grs$. Then by Lemma \ref{lem:inv1} iii) it follows $\grs \leq \tau'$, and by induction we get $\grs = \tau'$. Thus $\tau = s_\gra \circ \grs$, and by Lemma \ref{lem:inv2} we get $L(\tau) = L(\grs)+1$, a contradiction. Therefore it must be $\grs' < \grs$, hence $L(\grs')=L(\grs)-1$ and by Lemma \ref{lem:inv1} ii) we get $\grs' \leq \tau'$. Thus $\grs'=\tau'$ by induction, and it follows $\grs=\tau$.
\end{proof}

%%%%%%%%%%%%%%%%%%%%%%%%%%%%%%%%%%%%%%%
%%%%%%%%%%%%%%%%%%%%%%%%%%%%%%%%%%%%%%%
\subsection{Involutions in the Hermitian case}
%%%%%%%%%%%%%%%%%%%%%%%%%%%%%%%%%%%%%%%
%%%%%%%%%%%%%%%%%%%%%%%%%%%%%%%%%%%%%%%

Let $P$ be a parabolic subgroup of $G$ with abelian unipotent radical. In this subsection we will study the involutions of the form $\grs_{v(S)}$, 
where $v \in W^P$ and $S \subset \Phi^+(v)$ is an orthogonal subset.  Since $\Phi^+(v) \subset \Psi$, in particular $S$ will be an orthogonal subset of $\Psi$. In the following lemmas we collect some properties of such subsets which will be needed later on.

\begin{lem} \label{lem:panyushev}
Let $\gra \in \Phi^+_P$, and let $\grb, \grb' \in \Psi$ be orthogonal root, then:
\begin{enumerate}[\indent i)]
	\item $\grb\pm \grb'\notin \Phi$.
	\item if $\grb + \gra \in \Phi$, then $\grb' + \gra \not \in \Phi$.
	\item if $\grb - \gra \in \Psi$, then $\grb' - \gra \not \in \Psi$.
\end{enumerate}
\end{lem}

\begin{proof}
i). Notice that $\grb + \grb' \not \in \Phi$: otherwise $\grb + \grb' \in \Psi$ because $\gop^\mru$ is an ideal of $\gob$, and this cannot happen because $\gop^\mru$ is abelian. On the other hand, if $\grb - \grb' \in \Phi$, then $\grb + \grb' = s_{\grb'}(\grb-\grb')$ because $\grb,\grb'$ are orthogonal, thus $\grb + \grb' \in \Phi$ as well, which was already excluded. 

Claims ii) and iii) are taken from \cite[Lemma 1.2]{Pa}.
\end{proof}

If $S\subset \Psi$ is a set of orthogonal roots we define 
$$
\Phi_S=\mQ S\cap \Phi=\{\gra\in \Phi\st \grs_S(\gra)=-\gra\}.
$$

\begin{prp}	\label{prp:radici-attive-reali}
Let $S \subset \Psi$ be orthogonal, then
$$
\Phi_S = \{\gra \in \Phi \st \gra = \tfrac{1}{2}(\pm \grb \pm \grb') \quad \text{for some } \grb, \grb' \in S \}.
$$
\end{prp}

\begin{proof}
Let $\gra \in \Phi^+$. By the orthogonality of $S$ it follows that $\grs_S(\gra) = -\gra$ whenever $\gra$ has the required shape $\tfrac{1}{2}(\pm \grb \pm \grb')$. 
Suppose conversely that $\grs_S(\gra) = -\gra$. We distinguish two cases, depending on $\gra \in \Phi^+_P$ or $\gra \in \Psi$.

Suppose that $\gra \in \Phi^+_P$. By Lemma \ref{lem:panyushev} there are at most two roots in $S$ which are not orthogonal to $\gra$. 
Since $\grs_S(\gra) = \gra - \sum_{\grb\in S}\langle \gra,\grb^\vee\rangle\grb =-\gra$ and $\gra\notin S$,
we deduce that there are precisely two such roots $\grb, \grb'$. Consider the root 
$s_\grb(\gra)=\gra -\langle\gra, \grb^\vee\rangle\grb$: since $[\grd : \gra_P] \leq 1$ for all $\grd \in \Phi^+$, we must have 
$\langle\gra, \grb^\vee\rangle =\pm 1$, and similarly for $\grb'$. Moreover, by Lemma 
\ref{lem:panyushev} we can assume $\langle \gra, \grb^\vee \rangle =1$ and $\langle \gra, \grb'^\vee \rangle =-1$. 
Therefore $\grs_S(\gra) = \gra - \grb + \grb'$, and we get $\gra = \tfrac{1}{2}(\grb - \grb')$.

Suppose now that $\gra \in \Psi$. Then $\gra + \grb \not \in \Phi^+$ for all $\grb \in S$, hence  $\langle \gra, \grb^\vee \rangle \geq 0$. 
On the other hand $\grs_S(\gra) = -\gra = \gra - \sum_{\grb \in S} \langle \gra, \grb^\vee \rangle \grb$, thus by \eqref{eqn:Psi} we get that 
$\sum_{\grb \in S} \langle \gra, \grb^\vee \rangle = 2$. It follows that $2\gra = \grb + \grb'$ 
for some $\grb, \grb'\in S$. Therefore either $\gra \in S$, or $\gra$ is the half-sum of two such elements.
\end{proof}

\begin{prp}\label{prp:PhiS}
Let $S$ be an orthogonal subset of $\Psi$. 
\begin{enumerate}[\indent i)]
 \item Suppose that $\Phi$ is simply laced, then $\Phi_S=S\sqcup (-S)$;
 \item Suppose that $\Phi$ is not simply laced, and let $\Phi=\Phi_\ell\sqcup \Phi_s$ and $S=S_\ell\sqcup S_s$ be the partitions into long and short roots.
Then the following hold:
   \begin{enumerate}[\indent a)]
      \item $\Phi_{S}=\Phi_{S_\ell}\sqcup \Phi_{S_s}$; 
      \item $\Phi_{S_s}=S_s\sqcup (-S_s)$;
      \item $\Phi_{S_\ell}\cap \Phi_\ell =S_\ell\sqcup (-S_\ell)$;
      \item $S_\ell \sqcup(-S_\ell)=(\Phi_\ell\cap\Phi_S)$ and $S_s \sqcup (-S_s) = S_\ell^\perp \cap \Phi_S$;
      \item $\mZ S\cap \Phi^+=S$.
      \end{enumerate}
      \end{enumerate}
\end{prp}
\begin{proof}
Let $\gra \in \Phi^+_S$, then by the previous proposition we have either $\gra\in S$ or $\gra=\tfrac 12 (\grb\pm\grb')$ for two different roots $\grb,\grb'$ in $S$. Suppose that we are in the second case.

If $\grb$ and $\grb'$ are short, or if $\Phi$ is simply laced, then we would have $\|\gra\|^2 = \tfrac 12 \|\grb\|^2$, which is impossible. This implies i) and ii.b).

If $\grb$ is short and $\grb'$ is long, then we would have $\|\gra\|^2 = \tfrac 34 \|\grb\|^2$ which is also impossible, and similarly if $\grb$ is long and $\grb'$ is short. This implies ii.a).

If $\grb$ and $\grb'$ are long and the root system is not simply laced, then we have that $\|\gra\|^2 = \tfrac 12 \|\grb\|^2$, hence $\gra$ is a short root. This implies ii.c). 

Finally ii.d) and ii.e) follow from the other points.
\end{proof}

\begin{cor} \label{cor:uTvS}
Let $(u,R), (v,S)$ be admissible pairs, and suppose that $\grs_{u(R)} = \grs_{v(S)}$. Then $u(R) = v(S)$.
\end{cor}

\begin{proof}
Let $\grs=\grs_{u(R)}= \grs_{v(S)}$, and let $V$ be the corresponding eigenspace of eigenvalue $-1$. 
Then  $V\cap \Phi= u(\Phi_R)=v(\Phi_S)$. By points i) and ii.d) of Proposition \ref{prp:PhiS} we see that
$u(R)\cup (-u(R)) = v(S)\cup(-v(S))$. Thus $u(R) = u(\Phi_R) \cap \Phi^- = v(\Phi_S) \cap \Phi^- = v(S)$.
\end{proof}

\begin{oss}
The previous corollary is false if $P^\mru$ is not abelian. Suppose for example that $\Phi$ is of type $\sfD_4$, represented as usual in the Euclidean space $\mR^4$ with orthonormal basis $\{\gre_1, \gre_2, \gre_3, \gre_4\}$. Let $P$ be the maximal parabolic subgroup of $G$ corresponding to $\gra_2 = \gre_2 - \gre_3$, and let $u=v$ be the longest element in $W^P$.
If we choose $R = \{\gre_1 - \gre_4, \gre_1 + \gre_4, \gre_2 - \gre_3, \gre_2 + \gre_3\}$
and $S=\{\gre_1 - \gre_3, \gre_1 + \gre_3, \gre_2 - \gre_4, \gre_2 + \gre_4\}$, then $\grs_{u(R)}=\grs_{u(S)} = -\mathrm{id}$.
\end{oss}

%%%%%%%%%%%%%%%%%%%%%%%%%%%%%%%%%%%%%%%
%%%%%%%%%%%%%%%%%%%%%%%%%%%%%%%%%%%%%%%
\section{Parametrization of the $B$-orbits in $G/L$ and in $\gop^\mru$}\label{sez:par}
%%%%%%%%%%%%%%%%%%%%%%%%%%%%%%%%%%%%%%%
%%%%%%%%%%%%%%%%%%%%%%%%%%%%%%%%%%%%%%%

In this section we will describe the parametrization of the $B$-orbits in $\gop^\mru$ and in $G/L$. As already recalled, the parametrization of the $B$-orbits
in $\gop^\mru$ in terms of orthogonal subsets is due to Panyushev (see \cite[Theorem 2.2]{Pa}), whereas the parametrization of the $B$-orbits in $G/L$ is due to Richardson (see \cite[Theorem 5.2.4]{RS2}). Since the proof in \cite{RS2} is only sketched, we will include here complete proofs.

Consider the projection map $\pi:G/L\ra G/P$. Recall the decomposition $G/P=\bigsqcup_{v\in W^P}BvP$, and for $v \in W^P$ let $B^v=vPv^{-1}\cap B$ 
be the stabilizer of $vP\in G/P$ in $B$. Then
$$
\pi^{-1}(BvP/P) \isocan B\times^{B^v} vP/L.
$$
Hence we have a bijection between the $B$-orbits in $BvP/L$ and the $B^v$-orbits in $vP/L$, which is compatible with the Bruhat order.
Equivalently, if we set $B_v = v^{-1}B^v v = v^{-1}Bv \cap P$, then these orbits are also in bijection with the $B_v$-orbits in $P/L$.

\begin{lem}	\label{lem:B_w}
Let $v \in W^P$, then  $B_L = B_v \cap L$ and $B_v = v^{-1}Bv \cap B=B_L\ltimes U_v$ where $U_v$ is the subgroup of $P^\mru$ with
Lie algebra $\gou_v=\bigoplus_{\gra\in\Psi\senza\Phi^+(v)}\gou_\gra$.
\end{lem}

\begin{proof}
Notice that $B_v$ is the product of $v^{-1} B v \cap B$ with the root subgroups $U_\gra$ with $\gra \in \Phi^-_P \cap v^{-1}(\Phi^+)$, whereas $v^{-1} B v \cap B$ is the product of $T$ with the root subgroups $U_\gra$ with $\gra \in \Phi^+ \senza \Phi^+(v)$. Thus the claims follow because $v(\Phi^+_P) \subset \Phi^+$.
\end{proof}

We read now the action of $B_v$ on $P/L$, and more generally that of $P$ on $P/L$, as an action of $P$ on $\gop^\mru$. Let $\exp : \gop^\mru \ra P^\mru$ be the exponential map, as defined by Seitz  \cite[Proposition 5.3]{Se}. Then $\exp$ is a $P$-equivariant isomorphism, and setting $r_P(e)=\exp(e)L$ we get a $L$-equivariant isomorphism $r_P:\gop^\mru \ra P/L$. Notice that $r_P$ is not equivariant with respect to the action of $P$ on $P/L$ by left multiplication, if we consider the adjoint action of $P$ on $\gop^\mru$. In order to describe the action of $P$ on $\gop^\mru$ obtained from the isomorphism $r_P$, consider the following description of $P$:
$$ \begin{array}{ccc}
L \ltimes \gop^\mru & \stackrel{\sim}{\lra} & P \\
 (g,y) & \longmapsto & g \exp(y)
\end{array} $$
In particular, with this identification we have $B_v=B_L\ltimes \gou_v$. 
Using such description of $P$, we see that the action of $P$ on $\gop^\mru$ which makes $r_P$ into a $P$-equivariant isomorphism is given by
\begin{equation}	\label{eq:P-action}
	(g,y).x = \Ad_g(x+y),
\end{equation}
for all $(g,y) \in L\ltimes \gop^\mru$ and $x \in \gop^\mru$. 
We summarize the discussion in the following Lemma.

\begin{lem}\label{lem:trasferimento}
Let $v\in W^P$, then the map $B_v e\mapsto Bv\exp(e)L$ is an order isomorphism between the $B_v$-orbits in $\gop^\mru$ and the $B$-orbits in 
$BvP/L$. Moreover, $$\dim Bv\exp(e)L/L=\ell(v)+\dim B_v e.$$
\end{lem}

\begin{proof}
The first claim follows from the discussion above, and the second one follows from the equality $\dim  BvP/P=\ell(v)$.
\end{proof}

In order to provide some geometric background to the mentioned parametrizations, we recall how to attach a weight 
lattice to any algebraic variety acted by a connected solvable algebraic group.

\subsection{Standard base points}
If $Z$ is an algebraic variety acted by a connected solvable algebraic group $K$, recall (see e.g. \cite{Kn}) that 
we can associate to $Z$ a sublattice of $\calX(K)$, called the \textit{weight lattice} of $Z$ and defined as follows:
$$\calX_K(Z) = \{\text{weights of rational $K$-eigenfunctions } f \in \mk(Z)  \}.$$

Let $T_K \subset K$ be a maximal torus, then $K = T_K K^\mru$ and restriction of 
characters gives an identification $\calX(K) = \calX(T_K)$. If $Z$ is a homogeneous 
$K$-variety, we say that $z_0 \in Z$ is a $T_K$-\textit{standard base point} if 
$\Stab_{T_K}(z_0) \subset \Stab_{K}(z_0)$ is a maximal diagonalizable subgroup. Since a diagonalizable subgroup
is always conjugated to a subgroup of a maximal torus (see \cite[Corollary 6.3.6]{Springer}), standard base points always exist.
 
We have the following easy lemma (see e.g. \cite[Lemma 1.1]{GP}).

\begin{lem} \label{lem:B-rank}
Let $Z$ be a homogeneous $K$-variety and let $z_0 \in Z$ be a $T_K$-standard base point, 
then $T_K z_0 \subset Z$ is a closed $T_K$-orbit, and 
$\calX_{T_K}(T_K z_0) \subset \calX_{T_K}(T_K z)$ for all $z \in Z$. 
If moreover $H = \Stab_K(z_0)$, then $H = (T_K \cap H) H^\mru$ and
$$ \calX_K(Z) = \calX(K)^H = \calX_{T_K}(T_K z_0), $$
where $\calX(K)^H$ denotes the sublattice of $\calX(K)$ of the characters which are trivial on $H$.
\end{lem}

In the notation of the previous lemma, notice that a $T_K$-standard base point $z_0 \in Z$ is characterized 
by the equality $\calX_{T_K}(T_K z_0) = \calX_K(Z)$. 
Indeed, for all $z\in Z$ the restriction gives an inclusion $ \calX_K(Z) \subset \calX_{T_K}(T_K z)$, 
therefore the $T_K$-standard base points correspond to those points of $Z$ whose $T_K$-orbit has minimal weight lattice. 

The weight lattice is easily computed when $K = T_K$ is a torus acting rationally on a vector space $V$ and $Z$ is a $T_K$-orbit in $V$. Let $V = \bigoplus_{\chi \in \calX(T_K)} V_\chi$ be the isotypic decomposition of $V$ as a $T_K$-module. For $e \in V$, write $e = \sum_{\chi \in \calX(T_K)} e_\chi$ and denote
$$\supp(e) = \{\chi \in \calX(T_K) \st e_\chi \neq 0\}.$$

\begin{lem} \label{lem:lattice-T-orbit}
Let $V$ be a rational $T_K$-module, and let $e \in V$. Then $\calX_{T_K}(T_Ke) = \mZ \supp(e)$.
\end{lem}

\begin{proof}
Denote $Z = \ol{T_K e}$. Up to replace $V$ with a smaller submodule we may assume that $V = \bigoplus_{\chi \in \supp(e)} V_\chi$ and $V_\chi = \mk e_\chi$ for all $\chi \in \supp(e)$. Then $\mk[V]$ is generated by 
linear coordinates which are $T_K$-semiinvariant of weight $-\chi$, with $\chi\in \supp(e)$. Let $\grG_V$ (resp. $\grG_Z$) be the submonoid 
of $\calX(T_K)$ whose elements are the weights of the regular $T_K$-eigenfunctions on $V$ (resp. on $Z$). Since the weights of the coordinates of $V$ 
are precisely the elements of $-\supp(e)$, it follows that $\grG_V$ is generated as a monoid by $-\supp(e)$. By complete reducibility, 
every $T_K$-eigenfunction on $Z$ extends to a $T_K$-eigenfunction on $V$. Since no coordinate of $V$ vanishes on $e$, it follows that $\grG_Z = \grG_V$.
Thus the claim follows because $\calX_{T_K}(T_K e)$ is generated as a lattice by $\grG_Z$.
\end{proof}

\subsection{$B_v$-orbits in $\gop^\mru$ and $B$-orbits in $G/L$.}

We now enter into the parametrization of the $B_v$-orbits in $\gop^\mru$, 
and of the $B$-orbits in $G/L$. 
First we compute the standard base points for the action of $B_v$ on $\gop^\mru$
Recall that, if $S\subset \Psi$ is an orthogonal subset, 
we denoted $e_S=\sum_{\gra\in S}e_\gra$ and $e_\vuoto=0$. 

\begin{prp} \label{prop:punti-standard}
Let $(v,S), (v,R)$ be  admissible pairs, then the following hold.
\begin{enumerate}[\indent i)]
 \item The base point $e_S \in B_v e_S$ is $T$-standard, and $\calX_{B_v}(B_v e_S) = \mZ S$.
 \item The base point $vx_S \in Bvx_S$ is $T$-standard, and $\calX_{B}(B v x_S) = v(\mZ S)$.
 \item If $B_ve_R=B_ve_S$, then $R=S$.
\end{enumerate}
\end{prp}

\begin{proof}
i) Let $e_0 \in B_v e_S$ be a $T$-standard base point and denote $S_0 = \supp(e_0)$. 
Then by Lemmas \ref{lem:B-rank} and \ref{lem:lattice-T-orbit} we have $\calX_{B_v}(B_v e_S) = \calX_T(Te_0) = \mZ S_0$. 
Applying again Lemma \ref{lem:B-rank} it follows that $\mZ S_0 \subset \calX_T(Te_S) = \mZ S$, 
therefore by Proposition \ref{prp:PhiS} ii.e) we get the inclusion $S_0 \subset S$. 
Since $S_0$ is orthogonal, up to replace $e_0$ with some element in the same $T$-orbit we may assume that $e_0 = e_{S_0}$.

Given $e \in \gop^\mru$, consider the space
$$
	\mathrm{Lie}(B_v).e = [\got,e] + [\gou,e] + \gou_v,
$$
and notice that $\mathrm{Lie}(B_v).e$ and $\mathrm{Lie}(B_v).e'$ are conjugated if $e' \in B_v e$. Thus the dimension of $\mathrm{Lie}(B_v).e$ only depends on the orbit $B_v e$.

Since $S$ is orthogonal, by Lemma \ref{lem:panyushev} ii) we have
\begin{align*}
	[\got, e_{S_0}] 	= \bigoplus_{\gra \in S_0} \gou_\gra \subset & \bigoplus_{\gra \in S} \gou_\gra = [\got, e_S], \\
	[\gou,e_{S_0}] = \bigoplus_{\gra \in \Psi \cap (S_0+\Phi^+)} \gou_\gra \subset & \bigoplus_{\gra \in \Psi \cap (S +\Phi^+)} \gou_\gra = [\gou,e_S].
\end{align*}
On the other hand, by assumption $S$ is contained in $\Phi^+(v)$ and by Lemma \ref{lem:panyushev} i) we have 
$S\cap (S+\Phi^+) =\vuoto$, hence  $[\got, e_S] \cap ([\gou, e_S] + \gou_v) = 0$. 
Therefore, comparing the dimensions of $\mathrm{Lie}(B_v). e_S$ and $\mathrm{Lie}(B_v). e_0$, we get
$$ \card(S \smallsetminus S_0) + \dim ([\gou,e_S] + \gou_v) - \dim ([\gou,e_{S_0}] + \gou_v) = 0, $$
which yields $S_0 = S$.

ii) It is enough to show that $x_S$ is a $T$-standard base point in $(v^{-1}Bv)x_S$, regarded as a homogeneous variety for $v^{-1}Bv$. Notice that $\Stab_G(x_S) \subset P$ because $x_S \in P/L \subset G/L$. Since by definition $B_v = v^{-1}Bv \cap P$, it follows that
$$
	\Stab_{v^{-1}Bv}(x_S) = v^{-1}Bv \cap \Stab_P(x_S) = \Stab_{B_v}(x_S).
$$
We proved in i) that $e_S$ is a $T$-standard base point for $B_v e_S$. On the other hand by construction $\gop^\mru \ra P/L$ is a $B_v$-equivariant isomorphism, therefore $\Stab_T(x_S)$ is a maximal diagonalizable subgroup in $\Stab_{B_v}(x_S)$, thus the claim follows from the previous equality.

iii) Let $R,S$ be orthogonal subsets of $\Phi^+(v)$ such that $B_v e_R = B_v e_S$. Then by Proposition 
\ref{prop:punti-standard} i) it follows that $\mZ R=\mZ S$, hence $R=S$ by Proposition \ref{prp:PhiS} ii.e).
\end{proof}

In order to parametrize the $B_v$-orbits in $\gop^\mru$ we will proceed by induction on $\ell(v)$. 
The following lemma will be the key point of the inductive step. 

\begin{lem}	\label{lem:decomposizione-passo}
Let $(v,S)$ be an admissible pair, let $\gra \in \grD$ be such that $s_\gra v < v$ and denote $\grb = -v^{-1}(\gra) \in \Psi$. Then
$$
	B_{s_\gra v} e_S =
\left\{	\begin{array}{cc}
B_v e_S \sqcup B_v e_{S \cup \{\grb\}} & \text{if $S \cup \{\grb\}$ is orthogonal} \\ 
		B_v e_S & \text{otherwise}
		\end{array} \right.
$$
\end{lem}

\begin{proof}
Notice that, by Proposition \ref{prp:GP2}, $B_{s_\gra v} = B_L \ltimes (\gou_v\oplus \gou_\grb)=B_v\ltimes \gou_\grb$, hence
$$B_{s_\gra v} e_S = B_v \cdot \gou_\grb \cdot e_S = B_v(e_S + \mk e_\grb).$$
If $S \cup \{\grb\}$ is orthogonal, then $B_v(e_S + \mk^\times e_\grb) = B_v e_{S \cup \{\grb\}}$, thus the claim follows.

Suppose that $S \cup \{\grb\}$ is not orthogonal, and let $\grg \in S$ be such that $\langle \grb, \grg^\vee \rangle \neq 0$. 
Since $\grb + \grg \not \in \Phi$, we have $\grb - \grg \in \Phi$, hence from \eqref{eqn:Psi} it follows that $\grb - \grg \in \Phi_P$. 
On the other hand $\grb \in \Phi^+(v)$ is a maximal element by Proposition \ref{prp:GP2} i), thus $\grb - \grg \in \Phi^+_P$. 
Set $\grd = \grb - \grg$ and let $m \geq 1$ be the maximum such that $\grg + m \grd$ is a root, then by Lemma \ref{lem:panyushev} ii) we get that
$$
 u_\grb(t) \cdot e_S = e_S + c_1\, t\, e_\beta + \dots+c_m\, t^m\,e_{\grb + (m-1) \grd},
$$
for some $c_1, \ldots, c_m \in \mk^\times$. Since $\grb \in \Phi^+(v)$ is maximal, 
it follows that $e_S + \mk e_\grb \subset U_\grb e_S + \gou_v \subset B_v e_S$.
\end{proof}

We can now prove the following parametrization of the $B_v$-orbits in $\gop^\mru$. 

\begin{prp}	\label{prp:Bv-orbite}
Let $v \in W^P$, then the map $S \mapsto B_v  e_S$ induces a bijection
\[
\{S \subset \Phi^+(v) \st S \text{ is orthogonal }\} \longrightarrow \{ B_v \text{-orbits in } \gop^\mru\}
\]
\end{prp}

\begin{proof}
We have already proved in Proposition \ref{prop:punti-standard} iii) 
these orbits are all different.

We now show that every $B_v$-orbit in $\gop^\mru$ contains a point of the form $e_S$, 
for some orthogonal subset $S \subset \Phi^+(v)$. We proceed by induction on $\ell(v)$. 
If $\ell(v) = 0$, then $B_v = B$ and $\Phi^+(v) = \vuoto$. Since the $B$-action on $\gop^\mru$ defined by 
\eqref{eq:P-action} is transitive, the claim follows.

Suppose now that $\ell(v) > 0$ and let $\calO \subset \gop^\mru$ be a $B_v$-orbit. Let $\gra \in \grD$ be such that 
$s_\gra v < v$, and denote $u = s_\gra v$ and $\grb = -v^{-1}(\gra) \in \Psi$. Then $\Phi^+(v) = \Phi^+(u) \sqcup\{\grb\}$ 
and $B_u = B_v U_\grb$. By induction, there exists an orthogonal subset $S \subset \Phi^+(u)$ such that $B_u \calO = B_u e_S$.

If $S \cup \{\grb\}$ is not orthogonal, then $B_u e_S = B_v e_S$ by Lemma \ref{lem:decomposizione-passo}, thus $\calO = B_v e_S$ 
and the claim follows. Therefore we may assume that $S \cup \{\grb\}$ is orthogonal. Denote $S' = S \cup \{\grb\}$. Then by Lemma 
\ref{lem:decomposizione-passo} we have $B_u e_S = B_v e_S \cup B_v e_{S'}$, hence we have either $\calO = B_v e_S$ or $\calO = B_v e_{S'}$.
\end{proof}

As a consequence of the previous result we obtain both Panyushev's parametrization of the $B$-orbits in $\gop^\mru$ and Richardson's parametrization of the $B$-orbits in $G/L$.

Recall from the introduction that $V_L$ denotes the set of the admissible pairs. Moreover, for $S \subset \Psi$ orthogonal, we denote $g_S = \exp(e_S)$ and $x_S = g_S L/L$.

%If moreover $v \in W^P$, then we denote by $V_L(v)$ the set of the admissible pairs whose first component is $v$.

\begin{cor}	\label{cor:parametrizzazione}
\begin{enumerate}[\indent i)]
 \item Let $B$ act on $\gop^\mru$ via the adjoint action. Then the map $S \mapsto Be_S$ induces a bijection
\[
	\{S \subset \Psi \st S \text{ is orthogonal }\} \longrightarrow \{ B \text{-orbits in } \gop^\mru\}
\]
 \item The map $(v,S) \mapsto B v  x_S$ defines a bijection
\[
	V_L \longrightarrow \{\text{$B$-orbits in $G/L$}\}
\]
\end{enumerate}
\end{cor}

\begin{proof}
i). As it is abelian, $P^\mru$ acts trivially on $\gop^\mru$. 
Therefore every $B$-orbit in $\gop^\mru$ is actually a $B_L$-orbit. 
On the other hand, we have $\Phi^+(w^P) = \Psi$ and $B_{w^P} = B_L$, thus the claims follow by Proposition \ref{prp:Bv-orbite}.

ii). The claim follows from Proposition \ref{prp:Bv-orbite} and Lemma \ref{lem:trasferimento}.
\end{proof}

Notice that, if $S=\{\grb_1,\dots,\grb_m\}$ is an orthogonal subset of $\Psi$, then $T e_S \simeq \mk^\times e_{\grb_1}\times \dots \times \mk^\times e_{\grb_m}$.
Thus both the previous parametrizations are independent on the choice of the elements $e_\gra \in \gog_\gra$.

%%%%%%%%%%%%%%%%%%%%%%%%%%%%%%%%%%%%%%%
%%%%%%%%%%%%%%%%%%%%%%%%%%%%%%%%%%%%%%%
\subsection{The involution associated to a $B$-orbit.}	\label{ssec:invariants}
%%%%%%%%%%%%%%%%%%%%%%%%%%%%%%%%%%%%%%%
%%%%%%%%%%%%%%%%%%%%%%%%%%%%%%%%%%%%%%%

We now compute some other combinatorial invariants of the $B$-orbits in $G/L$ in 
terms of admissible pairs. First of all, by the equality $L = P \cap P^-$, we have two natural surjective maps
\begin{equation}	\label{eq:phi_-}
\varphi_+ : V_L \ra W^P, \qquad \qquad \varphi_- :  V_L \lra W^P
\end{equation}
respectively defined by projecting $B$-orbits in $G/L$ in $G/P$ and in $G/P^-$. Therefore $\varphi_+(v,S) = v$ and $\varphi_-(v,S) = \nu$, where $\nu \in W^P$ is defined by the equality $Bvg_S P^- = B\nu P^-$ and where $g_S = \exp(e_S)$.

Following Springer \cite{Sp}, we now recall how to attach an involution to any $B$-orbit in $G/L$, 
hence to any admissible pair. Let $\vartheta : G \ra G$ be the involution of $G$ such that $L = (G^\vartheta)^0$. 
Notice that $\vartheta$ acts trivially on $W$. Indeed it acts trivially on $T$, hence 
for $w\in W$ and $t\in T$ we have 
$wtw^{-1}=\vartheta(wtw^{-1})=\vartheta(w)\vartheta(t)\vartheta(w)^{-1}=\vartheta(w) t\vartheta(w)^{-1}$, which yields
$\vartheta(w)=w$. 

As in Springer \cite{Sp}, denote
$$\calV = \{g \in G \st g \vartheta(g)^{-1} \in N_G(T)\}.$$
Let $B$ and $L$ act on $G$ respectively by left and right multiplication, and consider the induced action of $B \times L$ on $G$. Then every $(B \times L)$-orbit in $G$ intersect $\calV$ in a $(T \times L)$-orbit (see the proof of \cite[Theorem 4.2]{Sp}). In this way we get a map
$\varphi_\calI : V_L \ra \calI,$ defined by setting 
\begin{equation}	\label{eq:phi_I}
\varphi_\calI(v,S) = g \vartheta(g)^{-1} T/T,
\end{equation}
where $g \in \calV$ is any element such that $Bvx_S = BgL/L$. In our case it is easy to describe explicitly these invariants. If $(v,S)$ is an admissible pair, we denote $g_{v(S)} = v g_S v^{-1}$.

\begin{lem}	\label{lem:twisted-involution}
Let $(v,S) \in V_L$, then $\varphi_\calI(v,S) = \grs_{v(S)}$ and $\varphi_-(v,S) = [v \grs_S]^P$.
Moreover $B\grs_{v(S)}B = B g_{v(S)}B$.
\end{lem}

\begin{proof}
Denote $g =  v \exp(-\tfrac{1}{2}f_S) \exp(e_S)$. Since $v(S) \subset \Phi^-$, 
notice that $Bvx_S = B g L/L$. On the other hand
$$
g \vartheta(g^{-1}) T/T = 
v \exp(-\tfrac{1}{2}f_S) \exp(2e_S) \exp(-\tfrac{1}{2}f_S)  v^{-1} T/T = 
v \sigma_S v^{-1} = \grs_{v(S)}.
$$
Therefore $g \in \calV$, and it follows that $\varphi_\calI(v,S) = \grs_{v(S)}$.

The second claim follows by noticing that
\[
	Bvx_S P^- = Bv \exp(-f_S) \exp(e_S) \exp(-f_S) P^- = Bv \grs_{S}P^-.	
\]

Finally, notice that $\grs_{v(S)} = (v \exp(-f_S) v^{-1}) g_{v(S)} (v \exp(-f_S) v^{-1})$. Since $(v,S)$ is admissible, we have by definition $v(S) \subset \Phi^-$. Thus $v \exp(-f_S) v^{-1} \in B$, and we get $\grs_{v(S)} \in B g_{v(S)} B$.
\end{proof}

%%%%%%%%%%%%%%%%%%%%%%%%%%%%%%%%%%%%%%%
%%%%%%%%%%%%%%%%%%%%%%%%%%%%%%%%%%%%%%%
\section{Dimension formulas, and the action of the minimal parabolic subgroups}\label{sez:azioneminimali}
%%%%%%%%%%%%%%%%%%%%%%%%%%%%%%%%%%%%%%%
%%%%%%%%%%%%%%%%%%%%%%%%%%%%%%%%%%%%%%%

In this section we will study the action of the minimal parabolic subgroups of $G$ on the set of the $B$-orbits in $G/L$. 
Let $(v,S) \in V_L$ and let $\gra \in \grD$. Since $B$ acts with finitely many orbits on $G/L$, 
in particular it acts with finitely many orbits on $P_\gra v x_S$. Following \cite{RS1} we define
$$ m_\gra (v,S) = \text{unique $(v',S') \in V_L$ s.t. $B v' x_{S'} \subset P_\gra v x_S$ is open} $$
In the previous notation, notice that, if $m_\gra (v,S) \neq (v,S)$, we have $\dim (Bv'x_{S'}) = \dim (Bvx_S) +1$. 

If $(v,S)\in V_L$ and $\gra\in \Delta$, we also set
$$ \calE_\gra(v,S) = \{(u,R) \in V_L \st m_\gra (u,R) = (v,S) \text{ and } (u,R) \neq (v,S)\}. $$

The following result due to Richardson and Springer will be needed in the proof of Theorem \ref{teo:panv}, which will constitute the basis of the induction to prove our main theorem.

\begin{lem}[{\cite[7.4]{RS1}}] \label{lem:azP1}
Let $(v,S) \in V_L$ and $\gra \in \Delta$, then the following hold.
\begin{enumerate}[\indent i)]
 \item If $m_\gra (v,S) = (v',S') \neq (v,S)$, then $\grs_{v'(S')} = s_\gra\circ \grs_{v(S)}>\grs_{v(S)}$. 
 \item $\calE_\gra(v,S)\neq \vuoto$ if and only if $s_\gra\circ \grs_{v(S)}<\grs_{v(S)}$.
 \item $\calE_\gra(v,S)=\vuoto$ if and only if $s_\gra\circ \grs_{v(S)}>\grs_{v(S)}$.
\end{enumerate}
\end{lem}

\begin{proof}
Since our statements are slightly different from those in \cite{RS1}, we provide some details.

i). Suppose that $m_\gra (v,S) \neq (v,S)$. Notice that $s_\gra \circ \grs_{v(S)}$ and $\grs_{v(S)}$ are always comparable by Lemma \ref{lem:inv0}. 
Assume by contradiction that $s_\gra\circ \grs_{v(S)} < \grs_{v(S)}$. Then $s_\gra \grs_{v(S)} < \grs_{v(S)}$ by Lemma \ref{lem:inv0} ii), 
therefore we get $\calE_\gra(v,S) \neq \vuoto$ by \cite[7.4(ii)]{RS1}. Let $(u,R) \in \calE_\gra(v,S)$, then by definition $Bvx_S$ is open in
$P_\gra u x_R$. On the other hand $P_\gra u x_R = P_\gra v x_S$, hence $m_\gra (v,S) = (v,S)$, a contradiction. This shows that $s_\gra\circ \grs_{v(S)}>\grs_{v(S)}$, and the remaining claim follows from \cite[7.4(i)]{RS1} together with Lemma \ref{lem:twisted-involution}.

ii). Suppose that $s_\gra \circ \grs_{v(S)} < \grs_{v(S)}$, then $s_\gra \grs_{v(S)} < \grs_{v(S)}$ by Lemma \ref{lem:inv0} ii), therefore we get $\calE_\gra(v,S) \neq \vuoto$ by \cite[7.4(ii)]{RS1}. Suppose now that $\calE_\gra(v,S)\neq \vuoto$, and let $(u,R) \in \calE_\gra(v,S)$. Then by construction we have $m_\gra (u,R) \neq (u,R)$, thus by i) we get $\grs_{v(S)} = s_\gra\circ \grs_{u(R)}>\grs_{u(R)}$. Therefore $s_\gra \circ \grs_{v(S)} = \grs_{u(R)} < \grs_{v(S)}$.

iii). Notice that $s_\gra \circ \grs_{v(S)}$ and $\grs_{v(S)}$ are never equal, and they are always comparable by Lemma \ref{lem:inv0}. Therefore the claim follows from ii).
\end{proof}

Using the techniques developed by Richardson and Springer \cite{RS1}, we can easily compute the dimension 
of a $B$-orbit (see \cite[Theorem 4.6]{RS1}).

Notice indeed that $0\in \ol{B_v e_S}$ for all $S \subset \Psi$. Thus, if $(v,S)\in V_L$ is associated to a closed orbit, it must be $S=\vuoto$. On the other hand
$\dim(Bvx_\vuoto) = \ell(v)+\dim \gou_v=\card(\Psi)$, therefore $(v,S)$ corresponds to a closed $B$-orbit if and only if $S = \vuoto$.
Suppose now that $(v,S)$ is an admissible pair with $S \neq \vuoto$: then $\grs_{v(S)}\neq \id$, and if $\gra\in \Delta$ is such that 
$s_\gra\circ \grs_{v(S)}<\grs_{v(S)}$, by Lemma \ref{lem:azP1} ii) we can find $(v',S')\in \calE_\gra(v,S)$. Thus by Lemma \ref{lem:inv2} we have
$$
\dim Bvx_S=\dim Bv'x_{S'}+1 \quad \mand \quad L(\grs_{v(S)})=L(\grs_{v'(S')})+1,
$$
and arguing by induction we obtain
\begin{equation}\label{eq:dim2}
 \dim Bvx_S= \card \Psi + L(\grs_{v(S)}).
\end{equation}
This is equivalent to formula \eqref{eq:dim1} in the Introduction.

\begin{dfn}
Let $(v,S) \in V_L$, the \textit{length} of $(v,S)$, denoted by $L(v,S)$, is the length 
$L(\grs_{v(S)})$ of the corresponding involution $\grs_{v(S)}$, namely
$$L(v,S) = \frac{\ell(\grs_{v(S)}) + \card S}{2}.$$
\end{dfn}

As a consequence  of the previous dimension formula
we also get a dimension formula for the adjoint $B$-orbits in $\gop^\mru$, which was conjectured 
by Panyushev \cite[Conjecture 6.2]{Pa}. Recall that $w^P$ denotes the longest element in $W^P$, $w_P$ the longest element in $W_P$ and $w_0=w^Pw_P$ the longest element in $W$.

\begin{cor} \label{cor:dim-formula}
Let $B$ act on $\gop^\mru$ via the adjoint action, and let $S \subset \Psi$ be an orthogonal subset. Then
$$ \dim B e_S = \frac{\ell(\grs_{w_P(S)}) + \card S}{2}. $$
\end{cor}

\begin{proof}
As we already noticed, every $B$-orbit in $\gop^\mru$ is a $B_L$-orbit, thus $\dim B e_S = \dim B_L x_S$. 
On the other hand $B_L = B_{w^P}$ and $\ell(w^P) = \card \Psi$. Hence by Lemma \ref{lem:trasferimento} we 
have $\dim Bw^P x_S=\card\Psi+\dim Be_S$,
and by formula \eqref{eq:dim2} it follows that
$\dim B e_S = \tfrac{1}{2}(\ell(\grs_{w^P(S)}) + \card S)$. 
To conclude the proof, it is enough to notice that $\ell(\grs_{w^P(S)}) = \ell(\grs_{w_P(S)})$: indeed $w_0^{-1} = w^P w_P$, 
therefore using the fact that $w^P$ and $w_0$ are involutions we get $\grs_{w_P(S)} = w_0 \grs_{w^P(S)} w_0^{-1}$.
\end{proof}

We give now a more precise result about the action of the minimal parabolic subgroups.

\begin{lem}\label{lem:azP2}
Let $(v,S) \in V_L$ and $\gra \in \Delta$, then the following hold.
\begin{enumerate}[\indent i)]
 \item If $\calE_\gra(v,S) \neq \vuoto$ and $s_\gra v<v$, then there exists 
$(v',S') \in \calE_\gra(v,S)$ with $v'=s_\gra v$.
 \item If $[s_\gra v]^P>v$, then $m_\gra (v,S) \neq (v,S)$.
\item If $\calE_\gra(v,S) = \vuoto$ and $m_\gra (v,S) = (v,S)$, then $[s_\gra v]^P = v$.
\end{enumerate}
\end{lem}

\begin{proof}
i). Notice that $B vx_S \subset P_\gra v x_S$ is open, because $\calE_\gra(v,S) \neq \vuoto$. Let $\calO = Bs_\gra v x_S$ and let $(u,R)$ be the corresponding admissible pair. Since $\calO \subset P_\gra v x_S$, we have $(u,R) \in \calE_\gra(v,S)$. On the other hand $s_\gra v \in W^P$, therefore $u = s_\gra v$. 

ii). Suppose that $m_\gra (v,S) = (u,R)$. Notice that $[s_\gra v]^P = s_\gra v$ because of the assumption. It follows that $P_\gra v x_S \cap B s_\gra v P/L$ is a dense open subset of $P_\gra v x_S$. Since $B$ has finitely many orbits on $G/L$, it has an open orbit $\calO \subset P_\gra v x_S \cap B s_\gra v P/L$. Then $\calO$ is open in $P_\gra v x_S$ as well, therefore $\calO = B u x_R$ and $u = s_\gra v$.

iii). By ii) we have $[s_\gra v]^P \leq v$, therefore either $s_\gra v < v$ or $[s_\gra v]^P = v$. Suppose that we are in the first case. Then $P_\gra v x_S \cap B s_\gra v P/L \neq \vuoto$, therefore there exists an admissible pair of the shape $(s_\gra v,S')$ such that $Bs_\gra v x_{S'} \subset P_\gra v x_S$. On the other hand $Bvx_S \subset P_\gra v x_S$ is open by the assumption, thus $(s_\gra v, S') \in \calE_\gra(v,S)$, a contradiction.
\end{proof}

Similarly to the previous lemma, we have the following analogous statements, obtained by looking at the representatives on $G/P^-$ rather than on $G/P$.

\begin{lem}\label{lem:azP3}
Let $(v,S) \in V_L$ and $\gra \in \Delta$, and set $\nu=[v\grs_S]^P$, then the following hold.
\begin{enumerate}[\indent i)]
 \item If $\calE_\gra(v,S) \neq \vuoto$ and $[s_\gra \nu]^P > \nu$, then there exists 
$(v',S') \in \calE_\gra(v,S)$ such that $[v'\grs_{S'}]^P = s_\gra \nu$.
 \item If $s_\gra \nu<\nu$, then $m_\gra (v,S) \neq (v,S)$.
 \item If $\calE_\gra(v,S) = \vuoto$ and $m_\gra (v,S) = (v,S)$, then $[s_\gra \nu]^P = \nu$.
\end{enumerate}
\end{lem}

\begin{proof}
i). Notice that $[s_\gra \nu]^P = s_\gra \nu$ and that $P_\gra vg_S L \cap B s_\gra \nu P^- \neq \vuoto$. Therefore by Lemma 
\ref{lem:twisted-involution} there exists $(v',S') \in V_L$ such that $[v'\grs_{S'}]^P = s_\gra \nu$. On the other hand 
$Bvx_S \subset P_\gra v x_S$ is open because $\calE_\gra(v,S) \neq \vuoto$, therefore $(v',S') \in \calE_\gra(v,S)$.

ii). Notice that $s_\gra \nu \in W^P$ and that $Bs_\gra \nu P^- \subset P_\gra \nu P^-$ is a dense open subset. Therefore the open $B$-orbit of $P_\gra v x_S$ is contained inside $P_\gra v x_S \cap Bs_\gra \nu P^-/L$, and it follows that $m_\gra (v,S) \neq (v,S)$.

iii). By ii) it holds $[s_\gra \nu]^P \geq \nu$, suppose that $[s_\gra \nu]^P > \nu$. Then $P_\gra v x_S \cap B s_\gra \nu P^-/L \neq \vuoto$ and there exists an admissible pair $(v',S')$ such that $Bv' x_{S'} \subset P_\gra v x_S \cap B s_\gra \nu P^-/L$. Since by assumption $Bvx_S \subset P_\gra v x_S$ is open, it follows that $(v',S') \in \calE_\gra(v,S)$, a contradiction. 
\end{proof}

Recall the following basic fact from \cite[\S 4.3]{RS1} (which is essentially based 
on the classification of the spherical subgroups of $\mathrm{SL}_2$): a $P_\gra$-orbit in $G/L$ decomposes 
at most into three $B$-orbits. In \cite[\S 4.3, Case C]{RS1} it is also proved that if $\calE_\gra(v,S)$ has two elements then
$\gra$ is {\em real} for $(v,S)$, namely $\grs_{v(S)}(\gra)=-\gra$.

In the following lemmas we further analyze the sets $\calE_\gra(v,S)$. First we will give conditions (which are only possible if the root system is of type $\sfB$ or $\sfC$) so that $\calE_\gra(v,S)$ contains a unique element, then we will characterize when $\calE_\gra(v,S)$ 
has cardinality 2.

\begin{lem}\label{lem:azP4a}
Let $(v,S)\in V_L$ and $\gra\in \Delta$ be such $v^{-1}(\gra)\in \Delta_P$. Denote $\grb= v^{-1}(\gra)$ and assume that there exists $\grg\in S$ such that $\grg-2\grb$ is also an element of $S$. Denote $\grg'=\grg-2\grb$ and $\grg_0=\grg-\grb$, and set $S'=S\senza \{\grg,\grg'\}$. 
Then $(v,S'\cup \{\grg_0\})\in V_L$ and 
  $$\calE_\gra(v,S)=\{(v,S'\cup \{\grg_0\})\}.  $$
\end{lem}

\begin{proof}
Since $\grg$ and $\grg-2\grb$ are both roots, $\grg_0$ must be a root as well, and since $\Phi$ is not of type $\sfG$, we also have $s_\grb(\grg)=\grg'$ and $\langle \grb^\vee,\grg\rangle=2$. Moreover we have $v(\grg_0)<0$, therefore setting $S_0=S'\cup \{\grg_0\}$ we get an admissible pair $(v,S_0)$.

Let $G_\grb\subset L$ be the subgroup generated by $U_\grb,U_{-\grb}$. Then $G_\grb$ is isomorphic to $\mathrm{SL}_2$ or to $\mathrm{PSL}_2$, and 
the vector space $V=\langle e_\grg,e_{\grg_0},e_{\grg'}\rangle$ is a representation of $G_\grb$ of highest weight $2$. 
By the construction of Chevalley groups (see e.g. \cite[Sections 1,2,3]{Steinberg}) we can normalize $e_\grg,e_{\grg_0},e_{\grg'}$ so that
\begin{align*}
\Ad_{ u_{\grb}(t)} (e_{\grg_0}) &= e_{\grg_0}+2te_{\grg}, &   \Ad_{u_{\grb}(t)} (e_{\grg'}) &= e_{\grg'}+te_{\grg_0}+t^2e_\grg, \\
\Ad_{ u_{-\grb}(t)} (e_{\grg_0}) &= e_{\grg_0}+2te_{\grg'}, &  \Ad_{u_{-\grb}(t)} (e_{\grg}) &= e_{\grg}+te_{\grg_0}+t^2e_{\grg'}.
\end{align*}

Notice that $xe_\grg+ye_{\grg'}+e_{S'}\in Te_S$ for all $x,y\neq 0$, thus $Bvx_S$ contains all the elements of the form
\begin{align*}
u_{\gra}(t)v\exp(xe_\grg + y e_{\grg'}+e_{S'})L & =
v\,u_{\grb}(t)\exp(xe_\grg + y e_{\grg'}+e_{S'})L  \\ 
&
= v\exp\big( \Ad_{u_{\grb}(t)}(xe_\grg + y e_{\grg'}+e_{S'})\big)L \qquad\qquad (\text{since }U_\grb\subset L) \\
&= v\exp\big((yt^2+x)e_\grg +yte_{\grg_0} +y e_{\grg'}+e_{S'})L .
\end{align*}
Therefore $Bvx_S$ contains all the elements of the form
$$
v\exp\big(a e_\grg + b e_{\grg_0} + c e_{\grg'}+e_{S'})L
$$
with $c\neq 0$ and $ca \neq b^2$. In particular it follows that $Bvx_{S'}$ is in the closure of $Bvx_S$. 
Similarly, since $\car \mk \neq 2$, all the elements of the form
$$
v\exp\big(a e_\grg + b e_{\grg_0} + e_{S'})L
$$
with $b\neq 0$ belong to $Bvx_{S_0}$.

We now analyze $P_\gra v x_S$. Write $P_\gra=BU_{-\gra}\sqcup Bs_\gra$. From the equality $s_\grb(\grg)=\grg'$ it follows
$$ s_\gra v x_S = vs_\grb g_S L=v g_{s_\grb(S)}s_\grb L= vg_SL =v x_S,$$
thus $B s_\gra vx_S=Bvx_S$. Let now $t\in \mk$, since $U_{-\grb}\subset L$ we have
\begin{align*}
u_{-\gra} (t) v x_S & = v  u_{-\grb}(t) \exp(e_S) L   
                      = v  \exp \big( \Ad_{u_{-\grb}(t)} (e_S) \big) L\\
                    & = v  \exp( e_{S'} + e_\grg + t e_{\grg_0}+(1+t^2)e_{\grg'})L 
\end{align*}
By the discussion above, it follows that $u_{-\gra} (t) v x_S$ is in $Bvx_{S_0}$ if $1+t^2=0$, and in $Bvx_S$ otherwise. The claim follows.
\end{proof}

\begin{lem}	\label{lem:azP4}Let $(v,S)\in V_L$, $\gra\in \Delta$ and set $\grb = -v^{-1}(\gra)$.
Then $\calE_\gra(v,S)$ has cardinality 2 if and only if $\grb \in S$, in which case
$$\calE_\gra(v,S) = \{(s_\gra v,S \senza \{\grb\}), (v,S \senza \{\grb\})\}.$$ 
\end{lem}

\begin{proof}
Assume first that $\grb \in S$. Then $s_\gra v\in W^P$, thus 
$(s_\gra v,S \senza \{\grb\})$ and $(v,S \senza \{\grb\})$ are both admissible pairs.
Set $S'=S\senza\{\grb\}$, then 
\begin{gather*}
v x_{S'}= v g_{S'}L=v u_\grb(-1)u_\grb(1)g_{S'} L= u_{-\gra}(-1)vg_S L, \\
s_\gra v x_{S'} 
=u_\gra(-1)\,u_\gra(1)\,s_\gra v g_{S'}L
=u_{\gra}(-1)\,s_\gra v u_{\grb}(1)\,  g_{S'}L
=u_{\gra}(-1)\,s_\gra v g_S L.
\end{gather*}
Thus both the previous elements belong to $P_\gra v x_S$, and $Bvx_{S'}$ and $B s_\gra v x_{S'}$ are both contained in $P_\gra v x_S$. 
Finally $\grs_{v(S')}=\grs_{s_\gra v(S')}=s_\gra \grs_{v(S)}=s_\gra\circ \grs_{v(S)}<\grs_{v(S)}$, therefore $(v,S'), (s_\gra v, S')\in 
\calE_\gra(v,S)$ and the claim follows. 

Suppose now that $\calE_\gra(v,S)$ has cardinality 2. As already recalled, by \cite[4.3]{RS1} it follows that $\grs_{v(S)}(\gra) = -\gra$. Notice that $[s_\gra v]^P\leq v$, otherwise $m_\gra(v,S)\neq (v,S)$ and $\calE_\gra(v,S)=\vuoto$. 

Suppose that $[s_\gra v]^P=v$, and set $\grb'=v^{-1}(\gra)$. Then $\grb'\in \Phi_P$, and since $v(\grb')=\gra$ is a simple root we obtain $\grb' \in \Delta_P$. Moreover $\grs_S(\grb')=-\grb'$, thus by Proposition \ref{prp:radici-attive-reali} we have $\grb'=\tfrac 12 (\grg-\grg')$ for some $\grg,\grg'\in S$. In particular $\grg'=\grg-2\grb'$, and by Lemma \ref{lem:azP4a} we deduce that $\card \calE_\gra(v,S) = 1$, a contradiction.

Therefore we have $[s_\gra v]^P<v$. Notice that $\grb\in \Psi$, that $v(\grb)=-\gra$ and that $\grs_S(\grb)=-\grb$. By Proposition \ref{prp:GP2} i) we see that $\grb$ is maximal in $\Phi^+(v)$, and by Proposition \ref{prp:radici-attive-reali} we have
$\grb=\tfrac 12 (\grg+\grg')$ for some $\grg,\grg'\in S$. Since $\grg$ and $\grg'$ are orthogonal, it follows that $\langle \grg,\grb^\vee\rangle = \langle \grg',\grb^\vee\rangle\neq 0$.
In particular, $\grb$ is comparable both with $\grg$  and $\grg'$, thus by the maximality of $\grb$ we get $\grb \geq \grg$ and $\grb\geq \grg'$. 
By the equality $\grb=\tfrac 12 (\grg+\grg')$ we get then $\grb = \grg = \grg'$, and in particular $\grb \in S$, which concludes the proof.
\end{proof}

Finally we will need the following property of the Bruhat order (which is equivalent to the one-step property of \cite{RS1}). It establishes some basic compatibilities between the Bruhat order on $G/L$ and the action of the minimal parabolic subgroups of $G$.

\begin{lem}	[{\cite[7.11(i) and 6.5]{RS1}}] \label{lem:ordG1}
Let $(u,R), (v,S) \in V_L$ be such that $Bux_R \subset \ol{Bvx_S}$. Let $\gra\in\Delta$ be such that $\calE_\gra(v,S) \neq \vuoto$ and let $(v',S') \in \calE_\gra(v,S)$, then the following hold.
\begin{enumerate}[\indent i)]
 \item If $m_\gra (u,R) \neq (u,R)$, then $P_\gra u x_R \subset \ol{Bvx_S}$ and there exists
       $(u',R') \in \calE_\gra(m_\gra(u,R))$ such that $Bu' x_{R'} \subset \ol{B v' x_{S'}}$.
 \item If $m_\gra(u,R) = (u,R)$, then either $Bux_R \subset \ol{Bv' x_{S'}}$ or there exists $(u',R') \in \calE_\gra(u,R)$ such that $Bu' x_{R'} \subset \ol{Bv' x_{S'}}$.
\end{enumerate}
\end{lem}

%%%%%%%%%%%%%%%%%%%%%%%%%%%%%%%%%%%%%%%
%%%%%%%%%%%%%%%%%%%%%%%%%%%%%%%%%%%%%%%
\section{The Bruhat order on abelian nilradicals}\label{sez:Pu}
%%%%%%%%%%%%%%%%%%%%%%%%%%%%%%%%%%%%%%%
%%%%%%%%%%%%%%%%%%%%%%%%%%%%%%%%%%%%%%%

In this section we study the Bruhat order of the $B_v$-orbits in $\gop^\mru$, where $B_v$ acts on $\gop^\mru$ via the action defined in \eqref{eq:P-action}. 
In order to do this, we will prove a more general statement characterizing the Bruhat order for some particular classes of $B$-orbits in $G/L$. 
Indeed, by the discussion in Section \ref{sez:par} (and more specifically by Lemma \ref{lem:trasferimento}),
the Bruhat order among the $B_v$-orbits in $\gop^\mru$ is equivalent to the 
Bruhat order among the $B$-orbits in $BvP/L \subset G/L$. The advantage of working inside $G/L$ is that there we can can use the action of the minimal parabolic subgroups, and argue by induction. As a consequence, we will get a characterization of the Bruhat order on the abelian nilradicals.

The next Lemma will give us the technical ingredients for the inductive step. Notice first that $Bv x_R \subset \overline{B v x_S}$ whenever 
$(v,R) \in \calE_\gra(v,S)$: indeed by definition we have 
$\overline{P_\gra v x_R} = \overline{Bv x_S}$, hence $Bvx_R \subset \overline{P_\gra v x_R} = \overline{Bv x_S}$.

\begin{lem}\label{lem:lem1}
Let $(v,R), (v,S) \in V_L$ with the same component on $W^P$. Suppose that $\grs_{v(R)} \leq \grs_{v(S)}$ and let $\gra\in \Delta$ be such that $\calE_\gra(v,S)\neq \vuoto$, 
then the following hold.
\begin{enumerate}[\indent i)]
  \item If $m_\gra (v,R) \neq (v,R)$, then $m_\gra (v,R) = (v,R')$ for some orthogonal subset $R' \subset \Phi^+(v)$, and $\grs_{v(R)} < \grs_{v(R')} \leq \grs_{v(S)}$.
  \item If $m_\gra (v,R) = (v,R)$ and $\calE_\gra(v,R) = \vuoto$, then there exists an orthogonal subset $S' \subset \Phi^+(v)$ such that $(v,S')\in \calE_\gra(v,S)$ and $\grs_{v(R)} \leq \grs_{v(S')} < \grs_{v(S)}$.
\item If $\calE_\gra(v,R) \neq \vuoto$, then there exist $(v',S') \in \calE_\gra(v,S)$ and an orthogonal subset $R' \subset \Phi^+(v')$ such that $(v',R') \in \calE_\gra(v,R)$ and $\grs_{v'(S')}\geq \grs_{v'(R')}$.
\end{enumerate}
\end{lem}

\begin{proof}
i). Because $\calE_\gra(v,S) \neq \vuoto$, notice that $m_\gra (v,S) = (v,S)$. Therefore by Lemma \ref{lem:azP2} ii) we have that $[s_\gra v]^P \leq v$. It follows that $BvP/L$ intersects $P_\gra v x_R$ in an open subset of the latter. Since $m_\gra (v,R) \neq (v,R)$, we get that $m_\gra (v,R) = (v,R')$ for some admissible pair of the shape $(v,R')$, and by Lemma \ref{lem:azP1} i) we get $\grs_{v(R)} < \grs_{v(R')} = s_\gra \circ \grs_{v(R)}$. On the other hand by Lemma \ref{lem:azP1} ii) we have $s_\gra \circ \grs_{v(S)} < \grs_{v(S)}$, thus by Lemma \ref{lem:inv1} iii) we get $\grs_{v(R')} \leq \grs_{v(S)}$.

ii). Let $(v',S') \in \calE_\gra(v,S)$, then $v' = v$ by Lemma \ref{lem:azP2} iii). On the other hand by Lemma \ref{lem:azP1} i) we have $\grs_{v(S')} = s_\gra \circ \grs_{v(S)} < \grs_{v(S)}$ and $s_\gra \circ \grs_{v(R)} > \grs_{v(R)}$, thus by Lemma \ref{lem:inv1} iii) we get $\grs_{v(R)} \leq \grs_{v(S')} < \grs_{v(S)}$.

iii). Let $(v',S') \in \calE_\gra(v,S)$ and $(v'',R') \in \calE_\gra(v,R)$. Notice that $[s_\gra v]^P \leq v$. If $[s_\gra v]^P = v$, then $v' = v''$. Otherwise $[s_\gra v]^P = s_\gra v < v$, and by Lemma \ref{lem:azP2} i) we can assume that $v' = v'' = s_\gra v$. By Lemma \ref{lem:azP1} i) we have $\grs_{v'(S')} = s_\gra \circ \grs_{v(S)}$ and $\grs_{v'(R')} = s_\gra \circ \grs_{v(R)}$. On the other hand by Lemma \ref{lem:azP1} ii) we have the inequalities $s_\gra \circ \grs_{v(R)} < \grs_{v(R)}$ and $s_\gra \circ \grs_{v(S)} < \grs_{v(S)}$, thus by Lemma \ref{lem:inv1} ii) we get $s_\gra \circ \grs_{v(R)} < s_\gra \circ \grs_{v(S)}$.
\end{proof}

We can now describe the Bruhat order on $BvP/L$, or equivalently the Bruhat order among the $B_v$-orbits in $\gop^\mru$.

\begin{teo}\label{teo:panv}
Let $(v,R), (v,S) \in V_L$. Then 
\begin{enumerate}[\indent i)]
 \item $B_v e_R \subset \overline{B_v e_S}$ if and only if $\grs_{v(R)}\leq \grs_{v(S)}$;
 \item $Bvx_R \subset \ol{B v x_S}$ if and only if $\grs_{v(R)} \leq \grs_{v(S)}$.
\end{enumerate}
\end{teo}

\begin{proof}
The two claims are equivalent by Lemma \ref{lem:trasferimento}. To prove the first implication, we use the formulation in $\gop^\mru$. 
Assume that $e_R \in \overline{B_v e_S}$. Since  $B_v e_S = B_L(e_S + \gou_v)$, by taking the exponential $\exp : \gop^\mru \ra P^\mru$ 
we get $g_R \in \ol{B_ L U_v g_S B_L} = \ol{B_v g_S B_L}$. Since $v B_v v^{-1}\subset B$, this implies that 
$g_{v(R)} \in \ol{B g_{v(S)} B}$,
hence $\grs_{v(R)} \leq \grs_{v(S)}$ by Lemma \ref{lem:twisted-involution}.

We now prove the second implication, by using the formulation in $G/L$. Assume that $\grs_{v(R)} \leq \grs_{v(S)}$.
We proceed by induction both on $L(\grs_{v(S)})$ and on $\ell(\grs_{v(S)}) - \ell(\grs_{v(R)})$. 
Suppose that $L(\grs_{v(S)}) = 0$, then $S = \vuoto$ and $\grs_{v(S)} = 1$, thus $\grs_{v(R)} = 1$ and $R = \vuoto$ as well. 
More generally, if $\ell(\grs_{v(R)}) = \ell(\grs_{v(S)})$, 
then $\grs_{v(R)} = \grs_{v(S)}$ and by Corollary \ref{cor:uTvS} it follows $R = S$.

Suppose now that $L(\grs_{v(S)}) > 0$ and $\ell(\grs_{v(R)}) < \ell(\grs_{v(S)})$. 
Since $\ell(\grs_{v(S)}) > 0$, by Lemma \ref{lem:azP1} ii) together with Lemma \ref{lem:inv0} ii), 
there exists $\gra \in \grD$ such that $\calE_\gra(v,S) \neq \vuoto$.

Suppose that $m_\gra (v,R) \neq (v,R)$ and set $m_\gra (v,R) = (v',R')$. 
Then by Lemma \ref{lem:lem1} i) we have $v' = v$, and $\grs_{v(R)} < \grs_{v(R')} \leq \grs_{v(S)}$. 
Therefore by induction we get $vx_{R'} \in \ol{B v x_S}$, and since 
$v x_R \in \ol{B v x_{R'}}$ it follows that $vx_R \in \ol{B v x_S}$ as well.

Suppose that $m_\gra (v,R) = (v,R)$ and $\calE_\gra(v,R) = \vuoto$. 
Then by Lemma \ref{lem:lem1} ii) there exists an admissible pair 
$(v',S')\in \calE_\gra(v,S)$ such that $v' = v$ and 
$\grs_{v(R)} \leq \grs_{v(S')} < \grs_{v(S)}$. 
Therefore $vx_R \in \ol{B v x_{S'}}$ by induction, and since $x_{S'} \in \ol{B v x_S}$ it follows $v x_R \in \ol{B v x_S}$ as well.

Suppose finally that $\calE_\gra(v,R) \neq \vuoto$. 
Then by Lemma \ref{lem:lem1} iii) there exist 
$(v',S') \in \calE_\gra(v,S)$ and $(v'',R') \in \calE_\gra(v,R)$ such that $v' = v''$ and $\grs_{v'(R')} \leq \grs_{v'(S')}$. On the other hand 
$L(\grs_{v(S)}) - L(\grs_{v'(S')})= \dim Bvx_S  - \dim Bv'x_{S'} = 1$, thus we get $B v'x_{R'} \subset \ol{B {v'} x_{S'}}$ by the inductive hypothesis. 
Applying $P_\gra$ to the previous inclusion we get then
$$Bvx_R \subset P_\gra v'x_{R'} \subset \ol{P_\gra v'x_{S'}} = \ol{Bvx_S},$$
and the proof is complete.
\end{proof}

In particular we get the following corollary, which was conjectured by Panyushev \cite[Conjecture 6.2]{Pa}.

\begin{cor}	\label{cor:panyushev-conj}
Let $B$ act on $\gop^\mru$ via the adjoint action.  
Suppose that $R, S \subset \Psi$ are orthogonal subsets, 
then $Be_R \subset \ol{B e_S}$ if and only if $\grs_{w_P(R)} \leq \grs_{w_P(S)}$.
\end{cor}

\begin{proof}
Since $P^\mru$ is abelian, every adjoint $B$-orbit in $\gop^\mru$ is a $B_L$-orbit. 
On the other hand, by Theorem \ref{teo:panv} i) applied to the case $v = w^P$, 
we have that $e_R \in \ol{B_L e_S}$ if and only if $\grs_{w^P(R)} \leq \grs_{w^P(S)}$. 
Now recall that $w_0=w^Pw_P$ and that $w_0$ and $w_P$ are involutions, hence
$\grs_{w^P(R)} = w_0 \grs_{w_P(R)} w_0^{-1}$ and similarly $\grs_{w^P(S)} = w_0 \grs_{w_P(S)} w_0^{-1}$.
Therefore the claim follows by noticing that the conjugation by $w_0$ preserves the Bruhat order.
\end{proof}

%%%%%%%%%%%%%%%%%%%%%%%%%%%%%%%%%%%%%%%
%%%%%%%%%%%%%%%%%%%%%%%%%%%%%%%%%%%%%%%
\section{The Bruhat order on Hermitian symmetric varieties}\label{sez:GL1}
%%%%%%%%%%%%%%%%%%%%%%%%%%%%%%%%%%%%%%%
%%%%%%%%%%%%%%%%%%%%%%%%%%%%%%%%%%%%%%%

We now come to main theorem of the paper. In this section and in the next 
one, we will prove a conjecture of Richardson and Ryan \cite[Conjecture 5.6.2]{RS3} describing the Bruhat order on $G/L$.

Recall from the Introduction the definition \eqref{eq:ordVL} of the partial order on $V_L$. Then we will prove the following theorem.

\begin{teo}\label{teo:conRR}
Let $(u,R), (v,S) \in V_L$. Then $Bux_R \subset \ol{Bvx_S}$ if and only if $(u,R) \leq (v,S)$.
\end{teo}

We start with a few remarks on the previous theorem.

\begin{oss}	\label{oss:deodhar}
Notice that one of the conditions involved in the definition of the combinatorial order among admissible pairs is always fulfilled. 
More precisely, if $(u,R) \in V_L$, then we always have $[u \grs_R]^P  \leq u$. 
If indeed $\grb\in R$ then $u(\grb) \in \Phi^-$, namely $us_\grb < u$. Thus by the orthogonality of $R$ we get $u \grs_R \leq u$.
\end{oss}

\begin{oss}	\label{oss:m_alpha-compatibile}
Let $(v,S) \in V_L$ and let $\gra \in \grD$, then the inequality $(v,S) \leq m_\gra (v,S)$ is always fulfilled. If indeed $v'$ 
is the longer of the two elements $v$ and $[s_\gra v]^P$, then $P_\gra v x_S \cap Bv'P/L$ is a dense open subset of $P_\gra v x_S$. Thus it must be $m_\gra(v,S) = (v',S')$ for some orthogonal subset $S' \subset \Phi^+(v')$, and by construction $v' \geq v$. Similarly, if $\mu = [v\grs_S]^P$ and $\nu$ is the shorter of the two elements $\mu$ and $s_\gra \mu$, then $P_\gra v x_S \cap B\nu P^-/L$ is a dense open subset of $P_\gra v x_S$, thus $[v'\grs_S']^P = \nu$ and we have $\nu \leq \mu$. Finally, the inequality $\grs_{v'(S')} \geq \grs_{v(S)}$ follows from Lemma \ref{lem:azP1} i).
\end{oss}

The first implication of Theorem \ref{teo:conRR} was already known. 
We recall the proof in the following lemma, which relies substantially on a result from \cite{RS3}.

\begin{lem}\label{lem:daGaC}
Let $(u,R), (v,S) \in V_L$, and suppose that $Bux_R \subset \ol{Bv x_S}$. Then $(u,R) \leq (v,S)$.
\end{lem}

\begin{proof}
By \cite[Lemma 1.1]{RS3} together with Lemma 
\ref{lem:twisted-involution} it follows that $\grs_{u(R)} \leq \grs_{v(S)}$. The 
inequality $u \leq v$ is obvious, and the inequality $[u\grs_R]^P \leq [v\grs_S]^P$ also is obvious thanks to 
Lemma \ref{lem:twisted-involution}. The last inequality follows from Remark \ref{oss:deodhar}.
\end{proof}

The other implication of Theorem \ref{teo:conRR} will be proved by induction. The next two lemmas will constitute the basis of the induction.

\begin{lem}\label{lem:baseind1}
Let $(v,R), (v,S) \in V_L$, then the following are equivalent:
\begin{enumerate}[\indent i)]
 \item $Bvx_R \subset \ol{Bv x_S}$;
 \item $(v,R) \leq (v,S)$;
 \item $\grs_{v(R)} \leq \grs_{v(S)}$.
\end{enumerate}
\end{lem}

\begin{proof}
The implication ii) $\Rightarrow$ iii) is trivial, iii) $\Rightarrow$ i) is the content of Theorem \ref{teo:panv} ii), and i) $\Rightarrow$ ii) follows from Lemma \ref{lem:daGaC}.
\end{proof}

\begin{oss}
 Notice that the previous proposition provides a geometric proof of the following combinatorial statement: if $(v,R), (v,S) \in V_L$ and $\grs_{v(R)} \leq \grs_{v(S)}$, then $[v\grs_S]^P\leq [v\grs_R]^P$. The same stament can be proved combinatorially in case 
 $v = [w_0]^Q$, where $Q$ is any parabolic subgroup of $G$, but we do not know a direct proof for $v\in W^P$. 
\end{oss}

\begin{lem}\label{lem:baseind2}
Let $(u,R), (v,S) \in V_L$ be such that $(u,R) \leq (v,S)$ and $L(u,R) \geq L(v,S)$. Then $(u,R) = (v,S)$.
\end{lem}

\begin{proof}
By assumption we have $\grs_{u(R)} \leq \grs_{v(S)}$ and $L(\grs_{v(S)}) \leq L(\grs_{u(R)})$. Therefore we get $\grs_{u(R)} = \grs_{v(S)}$ by Lemma \ref{lem:bruhat-Lcompatibile}, and $u(R) = v(S)$ by Corollary \ref{cor:uTvS}.

Since $\gro_P^\vee$ is fixed by $W_P$, by Proposition \ref{prp:GP1} iii) the inequalities $[v\grs_S]^P \leq  [u\grs_R]^P$ and $u \leq v$ are equivalent to the followings
$$
	v(\gro_P^\vee) \leq u(\gro_P^\vee), \qquad \qquad  u \grs_R (\gro_P^\vee) \leq v \grs_S(\gro_P^\vee).
$$

On the other hand, since $\gro_P^\vee$ is minuscule, we have
$$
	\grs_R(\gro_P^\vee) = \gro_P^\vee - \sum_{\gra \in R} \gra, \qquad \qquad
		\grs_S(\gro_P^\vee) = \gro_P^\vee - \sum_{\gra \in S} \gra.
$$
Since $u(R) = v(S)$, it follows that $u(\gro_P^\vee) - u\grs_R(\gro_P^\vee) = v(\gro_P^\vee) - v\grs_S(\gro_P^\vee)$, thus by the inequalities above we get $u(\gro_P^\vee) = v(\gro_P^\vee)$. Therefore $u = v$, and we also get $R=S$.
\end{proof}

\subsection{Strategy of the proof and preliminary lemmas}

The proof of Theorem \ref{teo:conRR} will be given in Section \ref{sez:GL2}. In this section we explain its structure and 
we prove some prelimary lemmas. 
In its general structure the proof  is  similar to that of Theorem \ref{teo:panv}, however there are some important differences 
which make the proof quite a bit more complicated. 

Let  $(u,R)$ and $(v,S)$ be two admissible pairs. As we have already seen in Lemma \ref{lem:daGaC},
if $Bux_R\subset \ol{Bvx_S}$ then the inequality $(u,R)\leq(v,S)$ follows from the 
work of Richardson and Springer. 
Assume now that $(u,R)\leq (v,S)$. Arguing by induction on the dimension of orbits as in the proof of Theorem \ref{teo:panv}, let 
$\gra \in \grD$ be such that $\calE_\gra(v,S)\neq \vuoto$. As in the proof of Theorem \ref{teo:panv}, we will analyze three main different cases.

The first case we analyze is when $\calE_\gra(u,R)=\vuoto$. The technical ingredients to deal with this case are contained in
Lemma \ref{lem:lem2} below. In the proof of the theorem, this corresponds to the cases $1_\gra$ and $2_\gra$. 

Thus we are reduced to the case where $\calE_\gra(u,R)\neq \vuoto$, and arguing by induction we would like to find admissible pairs 
$(u',R')\in \calE_\gra(u,R)$ and $(v',S')\in \calE_\gra(v,S)$ such that $(u',R')\leq(v',S')$. A posteriori this is indeed true, 
but we are not able to prove it directly in general. However we are able to find such pairs in most cases: the technical ingredient to do that 
is Lemma \ref{lem:lem3} below and this corresponds to the cases $3_\gra$ and $4_\gra$ in the proof of the theorem. 

There is a single case that remains outside of this analysis, namely the case where $\calE_\gra(u,R)$ has cardinality 2, and 
$\calE_\gra(v,S)=\{(v',S')\}$ with  $v'<v$ and $[v'\grs_{S'}]^P<[v\grs_{S}]^P$. Even more, we can assume that we are in this situation for all $\gra \in \grD$ such that $\calE_\gra(v,S) \neq \vuoto$. 

To treat this case we will argue by induction also on $\ell(v)-\ell(u)$. In particular, at the basis of our induction we will find the 
case where $u=v$, which was treated in Theorem \ref{teo:panv}. Notice that the inequality $\grs_{v(R)}\leq \grs_{v(S)}$ of Theorem \ref{teo:panv} is just one of those involved in the inequality $(v,R)\leq (v,S)$: however, as we have seen in Lemma \ref{lem:baseind1},
for these particular pairs this single inequality implies all the others. Most of the proof in Section \ref{sez:GL2} will be dedicated to treat this last case.\\ 

We now come to the preliminary lemmas mentioned above. The first one will be used to treat the case where $\calE_\gra(u,R)=\vuoto$.

\begin{lem}\label{lem:lem2}
Let $(u,R), (v,S) \in V_L$ with $(u,R) \leq (v,S)$, and  let $\gra \in \Delta$ be such that 
$\calE_\gra(v,S) \neq \vuoto$ and $\calE_\gra(u,R) = \vuoto$, then the following hold.
\begin{enumerate}[\indent i)]
 \item If $m_\gra (u,R) \neq (u,R)$, then $m_\gra(u,R) \leq (v,S)$.
 \item If $m_\gra (u,R) = (u,R)$ and $(v',S') \in \calE_\gra(v,S)$, then $(u,R) \leq (v',S')$.
\end{enumerate}
\end{lem}

\begin{proof}
Denote $\mu = [u \grs_R]^P$ and $\nu = [v \grs_S]^P$.

 i). Set $m_\gra(u,R) = (u',R')$ and denote $\mu' =  [u' \grs_{R'}]^P$. Notice that by the assumptions we have 
$Bu'P \subset P_\gra u P$ and $\ol{BvP} = \ol{P_\gra v P}$, thus $B u'P \subset \ol{BvP}$ because 
$P_\gra u P \subset \ol{P_\gra vP}$. Similarly, we have $B\mu' P^- \subset P_\gra \mu P^-$ and 
$\ol{B \nu P^-} = \ol{P_\gra \nu P^-}$, thus $B \mu' P^- \subset \ol{B\nu P^-}$ because 
$P_\gra \mu P^- \subset \ol{P_\gra \nu P^-}$. 
Together with Remark \ref{oss:deodhar}, this show the inequalities $\nu \leq \mu' \leq u' \leq v$. Finally by 
Lemma \ref{lem:azP1} iii) we get $\grs_{u'(R')} = s_\gra \circ \grs_{u(R)} > \grs_{u(R)}$, thus by Lemmas 
\ref{lem:azP1} ii) and \ref{lem:inv1} iii) we obtain $\grs_{u'(R')} \leq \grs_{v(S)}$ and the claim follows.

ii). Since $\calE_\gra(v,S) \neq \vuoto$, we have $m_\gra (v,S) = (v,S)$.  Therefore by Lemmas \ref{lem:azP2} 
ii) and \ref{lem:azP3} ii) we get $[s_\gra v]^P \leq v$ and $[s_\gra \nu]^P \geq \nu$. Similarly, by the same 
lemmas we have $[s_\gra u]^P = u$ and $[s_\gra \mu]^P = \mu$.

By the assumptions we have $\nu \leq \mu \leq u \leq v$ and $\grs_{u(R)} \leq \grs_{v(S)}$. Let 
$(v',S') \in \calE_\gra(v,S)$ and denote $\nu' = [v' \grs_{S'}]^P$, we have to show the inequalities 
$\grs_{u(R)} \leq \grs_{v'(S')}$, $u \leq v'$ and $\nu' \leq \mu$.

The first one follows from Lemma \ref{lem:azP1} ii)-iii) together with Lemma \ref{lem:inv1} iii): 
indeed $\grs_{u(R)} \leq \grs_{v'(S')} = s_\gra \circ \grs_{v(S)}$.

To show the second one, notice that either $v'=v$ or $v' = s_\gra v < v$. Thus, if $v \neq v'$, it 
must be $u < v$. Since $u < s_\gra u$, by the lifting property we get then $u \leq v'$.

To show the last inequality, notice that either $\nu' = \nu$ or $\nu' = s_\gra \nu > \nu$. In the 
first case there is nothing to show, suppose that $\nu' > \nu$. Then $s_\gra \nu \leq  s_\gra \mu$, thus $\nu' = s_\gra \nu = [s_\gra \nu]^P \leq [s_\gra \mu]^P = \mu$.
\end{proof}

The next two lemmas are related to the analysis of the case where both $\calE_\gra(v,S)$ and $\calE_\gra(u,R)$ are not empty. The first one is general and will be used many times, 
the second one, as explained above, gives the technical ingredient to treat many cases of the induction.

\begin{lem}\label{lem:lem2,5}
Let $(u,R), (v,S) \in V_L$ be such that $(u,R) \leq (v,S)$, and let $\gra \in \Delta$. If $(u',R') \in \calE_\gra(u,R)$ and 
$(v',S') \in \calE_\gra(v,S)$, then $\grs_{u'(R')} \leq \grs_{v'(S')}$.
\end{lem}

\begin{proof}
By Lemma \ref{lem:azP1} i) we have $\grs_{v'(S')} = s_\gra \circ \grs_{v(S)} < \grs_{v(S)}$, and similarly $\grs_{u'(R')} = s_\gra \circ \grs_{u(R)} < \grs_{u(R)}$ for all $(u',R') \in \calE_\gra(u,R)$. Thus $\grs_{u'(R')} \leq \grs_{v'(S')}$ by Lemma \ref{lem:inv1} ii).
\end{proof}

\begin{lem}\label{lem:lem3}
Let $(u,R), (v,S) \in V_L$ be such that $(u,R) \leq (v,S)$. Let $\gra \in \Delta$ be such that $\calE_\gra(u,R) \neq \vuoto$ and $\calE_\gra(v,S)\neq \vuoto$. Set $\mu = [u \grs_R]^P$ and $\nu = [v \grs_S]^P$, then the following hold.
\begin{enumerate}[\indent i)]
	\item Let $(v',S') \in \calE_\gra(v,S)$. Assume that either $v' = v$ or $[v' \grs_{S'}]^P = \nu$, then there exists $(u',R') \in \calE_\gra(u,R)$ such that $(u',R') \leq (v',S')$.
	\item Suppose that one of the following equalities holds:
$$[s_\gra v]^P = v,  \qquad [s_\gra \nu]^P = \nu, \qquad 
[s_\gra u]^P = u, \qquad [s_\gra \mu]^P = \mu.$$
Then, for all $(v',S') \in \calE_\gra(v,S)$, there exists $(u',R') \in \calE_\gra(u,R)$ such that $(u',R') \leq (v',S')$.
\end{enumerate}
\end{lem}

\begin{proof}
By assumption we have the inequalities 
$u \leq v$, $\nu \leq \mu$ and $\grs_{u(R)} \leq \grs_{v(S)}$. If $(u',R') \in \calE_\gra(u,R)$ and $(v',S') \in \calE_\gra(v,S)$, 
we will denote $\mu' = [u' \grs_{R'}]^P$ and $\nu' = [v' \grs_{S'}]^P$. Thus by Lemma \ref{lem:lem2,5} we have $(u',R') \leq (v',S')$ if and only if 
$u' \leq v'$ and $\nu' \leq \mu'$.

Since $\calE_\gra(u,R) \neq \vuoto$ and $\calE_\gra(v,S)\neq \vuoto$, we have 
$m_\gra(u,R) = (u,R)$ and $m_\gra(v,S) = (v,S)$. If moreover $(u',R') \in \calE_\gra(u,R)$, then we have the inequalities
$$
[s_\gra u]^P \leq u' \leq u, \qquad [s_\gra v]^P \leq v' \leq v, \qquad [s_\gra \mu]^P \geq \mu' \geq \mu, \qquad [s_\gra \nu]^P \geq \nu' \geq \nu.
$$

We prove i). Suppose that $v' = v$. Then $u' \leq u \leq v = v'$, thus we only have to show the inequality  $\nu' \leq \mu'$. 
We can choose $(u',R')  \in \calE_\gra(u,R)$ in such a way that $\mu' = [s_\gra \mu]^P$: if $[s_\gra \mu]^P = \mu$, 
then every $(u',R')  \in \calE_\gra(u,R)$ has this property, otherwise we can apply Lemma \ref{lem:azP3} i).
On the other hand we have $s_\gra \mu > \mu \geq \nu$, thus 
$s_\gra \mu \geq s_\gra \nu$. Therefore we get $\mu' = [s_\gra \mu]^P \geq [s_\gra \nu]^P \geq \nu'$.

Suppose that $\nu' = \nu$, and assume that $v' = s_\gra v < v$ (otherwise we apply the argument above). 
Since $\mu' \geq \mu$ we have $\nu' = \nu \leq \mu'$, thus we only have to show the inequality $u' \leq v'$.
Arguing as in the case $v=v'$, by Lemma \ref{lem:azP2} i), we can choose $(u',R')  \in \calE_\gra(u,R)$ in such a way that $u' = [s_\gra u]^P$. 
If $u' < u$, then $u' = s_\gra u < u$ and we get $u' = s_\gra u \leq s_\gra v = v'$. If instead $u' = u$, then $s_\gra u > u$, and by the lifting property we get $u < s_\gra v = v'$.

We prove ii). If $[s_\gra v]^P = v$ or $[s_\gra \nu]^P = \nu$, then any $(v',S') \in \calE_\gra(v,S)$ satisfies the condition in i), thus the claim follows from i). Therefore we can assume that $v' = s_\gra v < v$ and $\nu' = s_\gra \nu > \nu$.

Suppose that $[s_\gra u]^P = u$. Then we have $u' = u < s_\gra u$, thus $u' \leq v'$ by the lifting property. To show the other inequality, by Lemma \ref{lem:azP3} i) we can choose $(u',R') \in \calE_\gra(u,R)$ in such a way that $\mu' = [s_\gra \mu]^P$. Since $s_\gra \mu > \mu \geq \nu$, for such a choice we get $\mu' = [s_\gra \mu]^P \geq [s_\gra \nu]^P = \nu'$.

Suppose that $[s_\gra \mu]^P = \mu$. Notice that $s_\gra \mu > \mu$ and $s_\gra \nu > \nu$, thus $\mu' = [s_\gra \mu]^P \geq [s_\gra \nu]^P = \nu'$. To show the other inequality, by the previous case we can assume that $s_\gra u < u$. Thus by Lemma \ref{lem:azP2} i) we can choose $(u',R')  \in \calE_\gra(u,R)$ in such a way that $u' = s_\gra u$, and since $s_\gra u < u$ and $s_\gra v < v$ we get $u' = s_\gra u \leq  s_\gra v = v'$.
\end{proof}

\subsection{Proof of Theorem \ref{teo:conRR}}\label{sez:GL2}

By Lemma \ref{lem:daGaC} we have to show that, 
if $(u,R) \leq (v,S)$, then $Bux_R \subset \ol{Bvx_S}$.
Throughout this subsection, we will denote $\mu = [u \grs_R]^P$ and $\nu = [v \grs_S]^P$.

By Lemma \ref{lem:baseind2}, 
we have $L(v,S) - L(u,R) \geq 0$, therefore we can 
proceed by induction on $\ell(v)$, on $L(v,S)$, on $\ell(v) - \ell(u)$ and on $L(v,S) - L(u,R)$. 
By Lemmas \ref{lem:baseind1} and \ref{lem:baseind2}, 
the claim holds if $\ell(u) = \ell(v)$ or if $L(u,R) = L(v,S)$, and in particular if $\ell(v) = 0$ or if $L(v,S) = 0$. 
Therefore we may assume that $\ell(u) < \ell(v)$ and $L(u,R) < L(v,S)$.

Since $L(v,S) > 0$, we have $\ell(\grs_{v(S)}) >0$ as well. Thus by Lemma \ref{lem:inv0} ii) together with Lemma \ref{lem:azP1} ii) there exists $\gra \in \grD$ such that $\calE_\gra(v,S) \neq \vuoto$. 
Thanks to Lemma \ref{lem:lem2} we can easily conclude the proof in a few special cases.\\

\textit{Case }$1_\gra$. 
Suppose that $\calE_\gra(u,R) = \vuoto$ and $m_\gra(u,R) =(u',R')\neq (u,R)$. 
Then by Lemma \ref{lem:lem2} i) we have $(u,R) < (u',R') \leq (v,S)$. 
Since $L(v,S) - L(u',R') = L(v,S) - L(u,R)-1$, 
it follows from the inductive assumption on $L(v,S)-L(u,R)$ that $\ol{P_\gra u x_R}=\ol{Bu'x_{R'}} \subset \ol{Bvx_S}$, 
hence $B u x_R \subset \ol{Bvx_S}$.\\

\textit{Case }$2_\gra$.
Suppose that $\calE_\gra(u,R) = \vuoto$ and $m_\gra(u,R) = (u,R)$. Let $(v',S') \in \calE_\gra(v,S)$, then by Lemma \ref{lem:lem2} ii) we have $(u,R)\leq (v',S')$. On the other hand $L(v',S') - L(u,R) = L(v,S) - L(u,R)-1$, thus by the inductive assumption on $L(v,S)-L(u,R)$ it follows that $Bux_R \subset \ol{Bv' x_{R'}}$, and the claim follows because $Bv' x_{R'} \subset \ol{Bvx_R}$.\\

\textit{Case }$3_\gra$.
Suppose that $\calE_\gra(u,R) \neq \vuoto$, and assume that there exist $(u',R') \in \calE_\gra(u,R)$ and $(v',S') \in \calE_\gra(v,S)$ such that $(u',R') \leq (v',S')$. Then $L(v',S') = L(v,S)-1$, thus by the inductive assumption on $L(v,S)$ we get $Bu'x_{R'} \subset \ol{Bv' x_{S'}}$. Moreover, it follows that $P_\gra u'x_{R'} \subset \ol{P_\gra v' x_{S'}}$, hence $Bux_R \subset \ol{Bv x_S}$.\\

\textit{Case }$4_\gra$.
Suppose that $\calE_\gra(u,R) = \{(u',R')\}$ and let $(v',S') \in \calE_\gra(v,S)$. We claim that in this case we have necessarily $(u',R') \leq (v',S')$, so that we fall again in the previous case. Denote $\mu' = [u' \grs_{R'}]^P$ and $\nu' = [v' \grs_{S'}]^P$. By Lemma \ref{lem:lem2,5} we only have to show that $u' \leq v'$ and $\mu' \geq \nu'$. We only show the first inequality, as the other one is similar.

Notice that $[s_\gra v]^P \leq v' \leq v$, and that $u' = [s_\gra u]^P \leq u$  by the assumption on $(u,R)$. If $u' = u$ and $v'=v$ there is nothing to show, therefore we can assume that either $u' = s_\gra u < u$ or $v' = s_\gra v < v$. 

Suppose that we are in the first case: then either $v'=v$, in which case $u' < u \leq v'$, or $v' = s_\gra v < v$, in which case we get  $u' \leq v'$ by Lemma \ref{lem:parallelogramma} ii).

Suppose that we are in the second case, and not in the first case. Then we have $u'= u$ and $v' = s_\gra v$. Thus $u = [s_\gra u]^P$ (hence $u < s_\gra u$) and $s_\gra v < v$, and by Lemma \ref{lem:parallelogramma} iii) we get $u' = u \leq s_\gra v = v'$.\\

{\em Hence we can assume, and we will assume it from now on, that for all $\gra \in \grD$ such that 
$\calE_\gra(v,S)\neq \vuoto$ none of the previous cases hold}.\\

Since a $P_\gra$-orbit in $G/L$ decomposes 
at most into three $B$-orbits, it follows that $\calE_\gra(u,R)$ has cardinality 2 for all simple root $\gra$ such that $\calE_\gra(v,S)\neq \vuoto$ . 
Notice also that by Lemma \ref{lem:lem3} we are in the following setting.

\begin{claim} 	\label{lem:proof1}
Let $\gra \in \grD$ be such that $\calE_\gra(v,S) \neq \vuoto$, and let $(v',S') \in \calE_\gra(v,S)$. Denote $\nu' = [v' \grs_{S'}]^P$, then the following hold:
$$\begin{array}{lllll}
	{s_\gra u < u,} & \qquad & {s_\gra v < v,} & \qquad & {v' = s_\gra v,} \\
	{[s_\gra \mu]^P = s_\gra \mu,} & \qquad & {[s_\gra \nu]^P = s_\gra \nu,} & \qquad & {\nu' = s_\gra \nu.}
\end{array}$$
\end{claim}

Moreover, we have the following.

\begin{claim}	\label{lem:proof2}
 Let $\gra \in \grD$ be such that $\calE_\gra(v,S) \neq \vuoto$ and denote
$$\calE_\gra(u,R) = \{(u',R'), (u'',R'')\}.$$
Set $\grb = -u^{-1}(\gra)$, $\mu' = [u' \grs_{R'}]^P$ and $\mu'' = [u'' \grs_{R''}]^P$, then the following hold:
\begin{enumerate}	[\indent i)]
	\item $\grb \in R$, and up to switching $(u',R')$ and $(u'',R'')$ we have
\[
	\qquad \qquad u' = u, \qquad \mu' = s_\gra \mu , \qquad u'' = s_\gra u, \qquad \mu'' = \mu, \qquad R' = R'' = R \senza \{\grb\};
\]
	\item $u$ and $s_\gra  v$ are uncomparable;
	\item $\grb$ is maximal both in $\Phi^+(u)$ and in $\Phi^+(v)$, and $\grb = -v^{-1}(\gra)$.
\end{enumerate}
\end{claim}

\begin{proof}[Proof of the claim]
Point i) follows from Lemma \ref{lem:azP4}, 
by noticing that $u \grs_{u(R \senza \{\grb\})} = s_\gra u \grs_{u(R)}$ and $s_\gra u \grs_{u(R \senza \{\grb\})} = u \grs_{u(R)}$.

Point ii) follows from Lemma \ref{lem:lem2,5}, since $\nu' \leq \mu'$ and by the assumption above we cannot have $(u',R') \not \leq (v',S')$ otherwise 
we would be in step $3_\gra$.

We prove iii). Denote $\grb' =  -v^{-1}(\gra)$. 
Then $\grb$ is maximal in $\Phi^+(u)$ and $\grb'$ is maximal in $\Phi^+(v)$ thanks to Proposition \ref{prp:GP2} i) and the fact 
that $s_\gra v<v$ by Claim \ref{lem:proof1}.  
Notice that $\Phi^+(s_\gra u) = \Phi^+(u) \senza \{\grb\}$ and $\Phi^+(s_\gra v) = \Phi^+(v) \senza \{\grb'\}$. 
On the other hand, because $u \leq v$ and $u \not \leq s_\gra v$, by Proposition \ref{prp:GP1} we have 
$\Phi^+(u) \subset \Phi^+(v)$ and $\Phi^+(u) \not \subset \Phi^+(s_\gra v)$. Therefore it must be $\grb =\grb'$.
\end{proof}

Fix $\gra \in \grD$ such that $\calE_\gra(v,S) \neq \vuoto$, we will keep the notation of Claim \ref{lem:proof2}. 
By assumption we have $u < v$, thus, by the chain property (see \cite[Theorem 2.5.5]{BB}) and Lemma \ref{lem:GP1}, 
there exists $\gra_0 \in \grD$ such that $u \leq s_{\gra_0} v < v$ and (in particular) $s_{\gra_0} v \in W^P$. 
Let $\grb_0 = -v^{-1}(\gra_0)$. Then by Proposition \ref{prp:GP2} i) we see that $\grb_0$ is maximal in $\Phi^+(v)$. 
Since $u \leq s_{\gra_0} v$, notice that by 
Lemma \ref{lem:proof2} ii) it must be  $\calE_{\gra_0}(v,S) = \vuoto$. Denote
$$S_0 = S \cup \{\grb_0\}.$$

\begin{claim}	\label{lem:proof3}
The following hold:
\begin{enumerate}[\indent i)]
	\item $\gra_0$ and $\gra$ are orthogonal;
	\item $S_0 \subset \Phi^+(v)$ is an orthogonal subset, and $\grb_0 \not \in S$. Moreover,
	$m_{\gra_0}(v,S) = (v,S_0)$, and $\calE_{\gra_0}(v,S_0) = \{(s_{\gra_0} v, S), (v,S)\}$.
\end{enumerate}
\end{claim}

\begin{proof}[Proof of the claim]
i). Notice that $\gra$ and $\gra_0$ are orthogonal if and only if $\grb$ and $\grb_0$ are orthogonal. Since $\grb, \grb_0 \in \Psi$, we have $\grb + \grb_0 \not \in \Phi$. On the other hand $\grb$ and $\grb_0$ are both maximal in $\Phi^+(v)$ by Proposition \ref{prp:GP2} i), and by construction they cannot be equal. Thus $\grb - \grb_0 \not \in \Phi$, which shows that $\grb$ and $\grb_0$ 
are orthogonal.

ii). Since $s_{\gra_0} v < v$, it follows that $P_{\gra_0} v x_S \cap B s_{\gra_0} v P/L$ is a nonempty closed proper subset of $P_{\gra_0} v x_S$. On the other hand $\calE_{\gra_0}(v,S) = \vuoto$, thus $Bvx_S$ cannot be open in $P_{\gra_0} v x_S$. It follows that $m_{\gra_0}(v,S) \neq (v,S)$, and the complement of the open $B$-orbit in $P_{\gra_0} v x_S$ intersects both $BvP/L$ and $Bs_{\gra_0}vP/L$. Therefore $\calE_{\gra_0}(m_{\gra_0}(v,S))$ has cardinality at least 2, and the claims follow by Lemma \ref{lem:azP4}.
\end{proof}

\begin{claim}	\label{lem:proof4}
We have the following equalities:
$$
\calE_\gra(v,S) = \{(s_\gra v,S)\},  \quad \calE_\gra(s_{\gra_0}v,S) = \{(s_\gra s_{\gra_0}v, S)\}, \quad \calE_\gra(v,S_0) = \{(s_\gra v,S_0)\}.
$$
Moreover, $\calE_{\gra_0}(s_\gra v,S_0) = \{(s_\gra v,S), (s_\gra s_{\gra_0}v, S)\}$.
\end{claim}

\begin{proof}[Proof of the claim]
Recall that $\grb =- v^{-1}(\gra)$ from Claim \ref{lem:proof2} iii). Notice that $\grb \in \Psi \senza S$: otherwise by Lemma \ref{lem:azP4} 
there would exist $(v',S') \in \calE_\gra(v,S)$ with $v' = v$, and this is not possible by Claim \ref{lem:proof1}. On the other hand by construction 
$\grb \neq \grb_0$, therefore $\grb \not \in S_0$ as well.

It follows from Lemma \ref{lem:azP4} that $\calE_\gra(v,S)$, $\calE_\gra(s_{\gra_0}v,S)$ and $\calE_\gra(v,S_0)$ have all cardinality at most 1. 
On the other hand $(s_\gra v,S)$, $(s_\gra s_{\gra_0}v, S)$ and $(s_\gra v,S_0)$ are all admissible pairs, and we have the obvious inclusions
$$
	Bs_\gra v x_S  \subset P_\gra  v x_S, \qquad   Bs_\gra s_{\gra_0} v x_S \subset 
	P_\gra  s_{\gra_0} v x_S, \qquad Bs_\gra v x_{S_0} \subset P_\gra v x_{S_0}. 
$$
Therefore the first claim follows thanks to the inequalities $s_\gra v < v$ and $s_\gra s_{\gra_0} v < s_{\gra_0} v$, and the second claim follows from 
Lemma \ref{lem:azP4} thanks to the orthogonality of $\gra$ and $\gra_0$.
\end{proof}

\begin{claim}	\label{lem:proof5}
We have the inequality $(u,R) \leq (s_{\gra_0}v,S)$. In particular, we have the inclusion $B u x_R \subset \ol{B s_{\gra_0} v x_S}$.
\end{claim}

\begin{proof}[Proof of the claim]
Notice that $s_{\gra_0} \nu < \nu$, $[s_{\gra_0}v \grs_S]^P = [s_{\gra_0} \nu]^P = s_{\gra_0} \nu$ and that $\grs_{s_{\gra_0}v(S)} =  \grs_{v(S)}$. 
Thus by the inequality $(u,R) \leq (v,S)$ we immediately obtain the inequalities $s_{\gra_0} \nu < \mu$ and $\grs_{u(R)} \leq \grs_{s_{\gra_0}v(S)}$.

To prove the first claim, it only remains to show that $u \leq s_{\gra_0} v$. By Proposition \ref{prp:GP1} we have 
$\Phi^+(u) \subset \Phi^+(v)$. On the other hand $\Phi^+(s_{\gra_0}v)  = \Phi^+(v) \senza \{\grb_0\}$, and by construction we have 
$\grb_0 \not \in \Phi^+(u)$. It follows that $u \leq s_{\gra_0} v$, thus  $(u,R) \leq (s_{\gra_0}v,S)$, and the last claim follows 
from the inductive assumption on $\ell(v)$.
 \end{proof}

\begin{claim} \label{lem:proof6}
We have the inclusions
$$Bs_\gra u x_{R'} \subset \ol{B s_\gra s_{\gra_0}v x_S}, \qquad
B u x_R \subset \ol{B s_{\gra_0} v x_S}.$$
\end{claim}

\begin{proof}[Proof of the claim]
By Claim \ref{lem:proof5} we have the inclusion $Bux_R \subset \ol{B s_{\gra_0} v x_S}$, whereas by Claims \ref{lem:proof2} i) 
and \ref{lem:proof4} we have
$$\calE_\gra(u,R) = \{(s_\gra u,R'), (u,R')\}, \qquad \calE_\gra(s_{\gra_0}v,S) = \{(s_\gra s_{\gra_0}v, S)\}.$$
Thus by Lemma \ref{lem:ordG1} ii) either $Bs_\gra u x_{R'}$ or $Bu x_{R'}$ is contained in $\ol{B s_\gra s_{\gra_0}v x_S}$. 
On the other hand $s_\gra s_{\gra_0}v < s_\gra v$, and $u \not \leq s_\gra v$ thanks to Claim \ref{lem:proof2} ii). 
Therefore $u \not \leq s_\gra s_{\gra_0}v$, and we obtain $Bs_\gra u x_{R'} \subset \ol{B s_\gra s_{\gra_0}v x_S}$. 
The second inclusion follows as well, since $(u, R)= m_\gra(s_\gra u, R')$ and $(s_{\gra_0} v, S) = m_\gra(s_\gra s_{\gra_0} v, S)$.
\end{proof}

\begin{claim} \label{lem:proof7}
We have the inclusion $Bs_\gra u x_{R'} \subset \ol{Bs_\gra v x_S}$.
\end{claim}

\begin{proof}[Proof of the claim]
Since $s_\gra v < v$, by the inductive assumption on $\ell(v)$ it's enough to show that $(s_\gra u, R') \leq (s_\gra v, S)$.

Notice that $s_\gra u < s_\gra v$, indeed we have $u < v$ by assumption, and by Claim \ref{lem:proof1} 
we have $s_\gra u < u$ and $s_\gra v < v$. Notice also that by Lemma \ref{lem:lem2,5} we have 
$\grs_{s_\gra u(R')} \leq \grs_{s_\gra v(S)}$. Indeed $(s_\gra u, R') \in \calE_\gra(u,R)$ and 
$(s_\gra v,S) \in \calE_\gra(v,S)$, and by assumption $(u,R) \leq (v,S)$. Finally, notice that 
$[s_\gra u \grs_{R'}]^P = \mu$ and $[s_\gra v \grs_S]^P = s_\gra \nu$, therefore it only remains to show the inequality $\mu \geq s_\gra \nu$.

By the previous lemmas, the Hasse diagram of $(V_L,\leq)$ has the following subdiagram (for convenience we extend all admissible pairs with a 
third entry representing the Weyl group element associated by projecting on $G/P^-$):

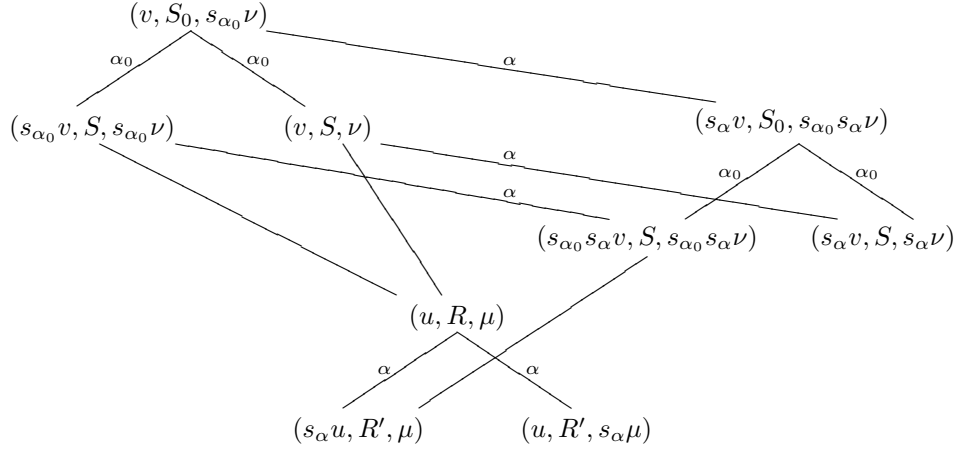
\[	\xy
(-225,0);(-240,-10);**\dir{-};
(-225,0);(-210,-10);**\dir{-};
(-224,2)*{(v,S_0,s_{\gra_0} \nu)};
(-234,-4)*{{}_{\gra_0}};
(-216,-4)*{{}_{\gra_0}};
(-183,-4)*{{}_{\gra}};
(-238,-13)*{(s_{\gra_0} v, S, s_{\gra_0} \nu)};
(-207,-13)*{(v,S,\nu)};
(-215,0);(-156,-9,5);**\dir{-};
(-237,-15);(-198,-35);**\dir{-};
(-227,-15);(-170,-25);**\dir{-};
(-205,-15);(-192,-35);**\dir{-};
(-200,-15);(-140,-25);**\dir{-};
(-145,-15);(-160,-25);**\dir{-};
(-145,-15);(-130,-25);**\dir{-};
(-183,-16.5)*{{}_{\gra}};
(-183,-21.5)*{{}_{\gra}};
(-154,-19)*{{}_{\gra_0}};
(-136,-19)*{{}_{\gra_0}};
(-146,-12)*{(s_\gra v, S_0 , s_{\gra_0} s_\gra \nu)};
(-165,-27.5)*{(s_{\gra_0} s_\gra v, S,  s_{\gra_0} s_\gra \nu)};
(-134,-27.5)*{(s_\gra v, S, s_\gra \nu)};
(-190,-40);(-205,-50);**\dir{-};
(-190,-40);(-175,-50);**\dir{-};
(-195,-50);(-165,-30);**\dir{-};
(-180,-45)*{{}_{\gra}};
(-199.5,-45)*{{}_{\gra}};
(-190,-38)*{(u,R,\mu)};
(-203,-53)*{(s_\gra u ,R', \mu)};
(-173,-53)*{(u ,R', s_\gra \mu)};
\endxy
\]	

Consider now the Hasse diagram in $W^P$ with the Bruhat order obtained by looking at the last components of the entries of 
the previous diagram (notice that this reverses the order). By assumption we have the three inequalities $\nu \leq \mu$, 
$\nu< s_\gra \nu$ and $\mu < s_\gra \mu$, therefore $s_\gra \nu \leq s_\gra \mu$ as well. Thus we get the following diagram in $W^P$:
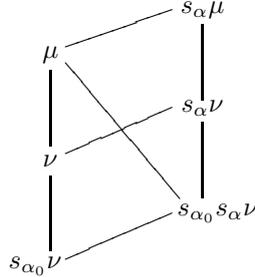
\[	\xy
(-20,-3)*{\mu};
(-20,-5);(-20,-15);**\dir{-};
(-20,-17)*{\nu};
(-20,-19);(-20,-29);**\dir{-};
(-22,-31)*{s_{\gra_0} \nu};
(0,3)*{s_\gra \mu};
(0,1);(0,-9);**\dir{-};
(0,-10)*{s_\gra \nu};
(0,-12);(0,-22);**\dir{-};
(2,-24)*{s_{\gra_0} s_\gra \nu};
(-18,-2);(-4,3);**\dir{-};
(-18,-16);(-4,-10);**\dir{-};
(-18,-30);(-4,-24);**\dir{-};
(-18,-4);(-3,-22);**\dir{-};
\endxy
\]	
\\

By applying the elements in the diagram above to the fundamental coweight $\gro_P^\vee$ as in Proposition \ref{prp:GP1}, 
we can translate the previous diagram into the following Hasse diagram in the coweight lattice with the dominance order:
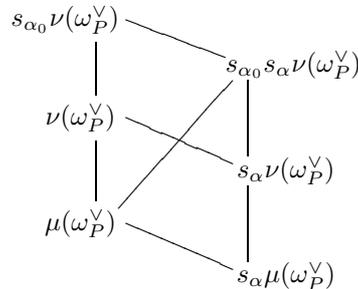
\[	\xy
(-24,3)*{s_{\gra_0} \nu(\gro_P^\vee)};
(-20,0);(-20,-7);**\dir{-};
(-22,-10)*{\nu(\gro_P^\vee)};
(-20,-13);(-20,-21);**\dir{-};
(-22,-24)*{\mu(\gro_P^\vee)};
(6,-3)*{s_{\gra_0} s_\gra \nu(\gro_P^\vee)};
(0,-5);(0,-15);**\dir{-};
(5,-17)*{s_\gra \nu(\gro_P^\vee)};
(0,-19);(0,-29);**\dir{-};
(5,-31)*{s_\gra \mu(\gro_P^\vee)};
(-3,-2);(-16,3);**\dir{-};
(-2,-16);(-16,-10);**\dir{-};
(-2,-30);(-16,-24);**\dir{-};
(-2,-5);(-17,-22);**\dir{-};
\endxy
\]

Since $\gro_P^\vee$ is minuscule, we have the equalities
\begin{gather*}
\gra = s_{\gra_0} \nu(\gro_P^\vee) - s_{\gra_0} s_\gra \nu(\gro_P^\vee) = \nu(\gro_P^\vee) - s_\gra\nu(\gro_P^\vee),\\
\gra_0 = s_{\gra_0} \nu(\gro_P^\vee) - \nu(\gro_P^\vee).
\end{gather*}
Therefore
$$
	\nu(\gro_P^\vee) - \mu (\gro_P^\vee) = \big(s_{\gra_0} \nu(\gro_P^\vee) - \mu (\gro_P^\vee) \big) - 
	\big(s_{\gra_0} \nu(\gro_P^\vee) - \nu(\gro_P^\vee) \big) \geq \gra - \gra_0 .
$$
On the other hand by assumption we have $\mu (\gro_P^\vee)  \leq \nu(\gro_P^\vee)$, and by construction 
$\gra_0$ and $\gra$ are distinct simple roots. Therefore it must be $\nu(\gro_P^\vee) - \mu (\gro_P^\vee) \geq \gra$, 
and we get $\mu(\gro_P^\vee) \leq s_\gra \nu(\gro_P^\vee)$. Thus by Proposition \ref{prp:GP1} we get $s_\gra \nu \leq \mu$ .
\end{proof}

Since $(v,S) = m_\gra(s_\gra v, S)$, we have the equality $\ol{Bvx_S} = \ol{P_\gra s_\gra v x_S}$. On the other hand by Claim \ref{lem:proof7} 
we have $Bs_\gra u x_{R'} \subset \ol{Bs_\gra v x_S}$, therefore we get
$$
	Bux_R \subset P_\gra s_\gra u x_{R'} \subset \ol{P_\gra s_\gra v x_S} = \ol{Bvx_S}.
$$
The proof is complete.

\subsection{An example: the case of $\mathbf{Sp_4}$} \label{ex:Sp4}
We now give some details about the example of $G=\mathrm{Sp}_4$, and explain the failure of our theorems in characteristic 2.

Let $\gra_1,\gra_2$ be the simple roots of $G$, and denote $\eta=\gra_1+\gra_2$ and $\theta=2\gra_1+\gra_2$. In this case $P=P_{\gra_1}$ and $L=\mathrm{GL}_2$.
First we will give the Hasse diagram for $G/L$ in characteristic different from 2, then we will explain how the situation changes in characteristic 2. 

If $\car \mk \neq 2$, the Bruhat order on $G/L$ is described by the picture here below, which is organized as follows. The vertical diagram on the right and the horizontal diagram on the bottom are respectively the Hasse diagrams of the $B$-orbits in $G/P^-$ and in $G/P$. 
The other diagram is the Hasse diagram of the $B$-orbits in $G/L$. For every entry of the Hasse diagram of $G/L$ we have written down the set $S$. One can read the projection of a $B$-orbit in $G/L$ on $G/P$ by looking at the corresponding entry on the same column in the horizontal diagram on the bottom, and similarly 
one can read the projection on $G/P^-$ by looking at the corresponding entry on the same row in the vertical diagram on the right.
There is only one exception to these rules: the entry on the upper-left corner of the Hasse diagram of $G/L$, which is the open 
$B$-orbit in $G/L$ and which projects onto the open orbits both on $G/P$ and on $G/P^-$, which is not aligned with the two projections for graphical reasons.
Finally, we have decorated the arrows representing the covering relations in the Hasse diagrams with the number $1$ (resp. $2$) when the two $B$-orbits are in the same $P_{\gra_1}$ orbit (resp. in the same $P_{\gra_2}$ orbit).

$$
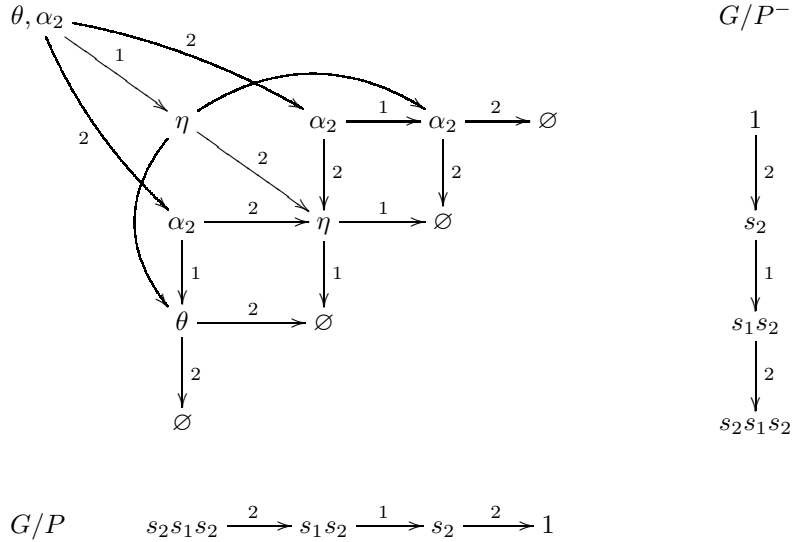
\xymatrix{
\theta,\gra_2 \ar^1[rd] \ar^2@/^/[rrd] \ar_2@/_/[rdd] & & & & & & G/P^-\\
 & \eta  \ar^2[rd] \ar@/_1.5pc/[dd] \ar@/^1.5pc/[rr] & \gra_2 \ar^2[d] \ar^1[r]  & \gra_2 \ar^2[r] \ar^2[d] & \vuoto &  &        1 \ar^2[d]\\
 & \gra_2 \ar^1[d] \ar^2[r]                          & \eta \ar^1[r] \ar^1[d]    & \vuoto                   &        &  &      s_2 \ar^1[d]\\
 & \theta \ar^2[d] \ar^2[r]                          & \vuoto                    &                          &        &  &  s_1 s_2 \ar^2[d]\\
 & \vuoto                                            &                           &                          &        &  &        s_2s_1s_2 \\
 G/P& s_2s_1s_2 \ar^2[r]  & s_1s_2 \ar^1[r] & s_2 \ar^2[r] & 1      }
$$

\medskip

The proof given in Section \ref{sez:GL2} is of course very fast in this case. The diagram above represents the combinatorial order on the set $V_L$ of the admissible pairs, and we want 
to prove that it represents the Bruhat order on $G/L$. The statement is trivial for all covering relations that are decorated. We only have two covering relations that are not decorated,
both starting from the admissible pair $(v,S) = (s_2s_1s_2, \{\eta\})$. Following the general line of the proof given in the previous section, suppose that we want to prove that the closure of the corresponding orbit contains the orbit defined by the admissible pair $(u,R) = (s_2, \{\gra_2\})$, then we argue by induction by noticing that $\calE_{\gra_2}(v,S)=\{(s_1s_2,S)\}$ and $\calE_{\gra_2}(u,R)=\{(s_2,\vuoto),(1,\vuoto)\}$, and that 
$(s_1s_2,S)\geq (s_2,\vuoto)$.

\medskip

Suppose now that $\car \mk = 2$. Denote $J=\left(\begin{smallmatrix}0&I\\I&0\end{smallmatrix}\right)$, and notice that $G = \{ g \in \mathrm{GL}_4 \st  gJg^t=J \}$
is a simply connected algebraic group of type $\sfC_2$. 
So, as in the case of characteristic different from 2, the nilradical $\gop^\mru$ is the algebra of symmetric matrices and $L=\mathrm{GL}_2$. 

If $\car \mk \neq 2$, the Bruhat order of the $B$-orbits in $\gop^\mru$ is represented in the diagram of $G/L$ by the subdiagram of the elements 
whose projection in $G/P$ equals $s_2s_1s_2$: in particular we have five $B$-orbits in $\gop^\mru$. 
A direct computation shows that, if $\car \mk = 2$, there is an new $B$-orbit in $\gop^\mru$ and that is $Be_S$ with  
$S=\{\eta,\theta\}$, and the Bruhat order on $\gop^\mru$ is represented by the following Hasse diagram:\\

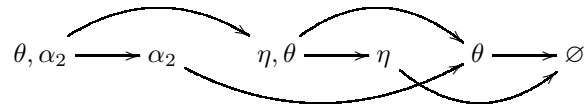
$$
\xymatrix{
\theta,\gra_2 \ar[r] \ar@/^1.5pc/[rr] & \gra_2 \ar@/_1.5pc/[rrr] & \eta, \theta \ar[r] \ar@/^1.5pc/[rr] & \eta \ar@/_1.5pc/[rr] & \theta \ar[r] & \vuoto
}
$$
\medskip

\noindent In particular, the Bruhat order on $\gop^\mru$ in characteristic different from 2 is different form that in characteristic 2, and similarly for the Bruhat order on $G/L$.

\end{document}